\documentclass[12pt]{amsart}
\usepackage{amsmath}
\usepackage{amssymb}
\usepackage{amsfonts}
\usepackage{mathrsfs}
\usepackage{graphicx}
\usepackage{epsfig}
\usepackage[T1]{fontenc}
\numberwithin{equation}{section}       

\theoremstyle{plain}
\newtheorem{theorem}{Theorem}[section]

\newtheorem{prop}{Proposition}[section]

\newtheorem{coro}[prop]{Corollary}
\newtheorem{lemma}[prop]{Lemma}

\newtheorem{definition}[prop]{Definition}

\newtheorem{remark}[prop]{Remark}
\theoremstyle{remark}
\newtheorem{exam}[prop]{Example}

\newtheoremstyle{citing}
  {3pt}
  {3pt}
  {\itshape}
  {}
  {\bfseries}
  {.}
  {.5em}
  {\thmnote{#3}}

\theoremstyle{citing}

\DeclareMathAlphabet{\mathpzc}{OT1}{pzc}{m}{it} 

\newcommand{\C}{\mathbb{C}}
\newcommand{\D}{\mathbb{D}}

\newcommand{\N}{\mathbb{N}}

\newcommand{\R}{\mathbb{R}}

\newcommand{\Z}{\mathbb{Z}}

\newcommand{\teta}{\widetilde{\teta}}

\newcommand{\eps}{\varepsilon}

\newcommand{\dist}{d}

\DeclareMathOperator{\id}{Id}
\DeclareMathOperator{\diam}{diam}

\begin{document}

\title[]{Polynomial-like dynamics of analytic maps}

\author{Genadi Levin}

\address{Institute of Mathematics, The Hebrew University of Jerusalem, Givat Ram,
Jerusalem, 91904, Israel}

\email{levin@math.huji.ac.il}

\date{\today}

\maketitle

\begin{abstract}

The theory of polynomial-like maps is of fundamental importance in holomorphic dynamics.
We study dynamical properties of a larger class of maps.
Our main result is that, under some natural assumptions, a map of this class has a completely invariant compact set
if and only if this set is the filled Julia set of a polynomial-like restriction of the map.
We also generalize this result to include maps with non-connected domains of definition.

\end{abstract}

\section{Introduction and main results}


Polynomial-like maps were introduced by Douady and Hubbard \cite{DH} to explain, in particular, the fractal nature of non-linear phenomena in holomorphic dynamics.
In this paper proper holomorphic map which are not polynomial-like are considered and we are interested in conditions that guarantee a polynomial-like (PL) restriction of such a map. We prove in Theorem \ref{thm-m} that, in the class of BI maps, see Definition \ref{d-bi}, a PL restriction exists if and only if the map has a completely invariant compact set $K$ (under some natural restrictions).
This gives a rather complete description of the dynamics of BI maps on such a compact and its neighborhood.
It would be interesting to study dynamical properties of locally analytic self-maps of a compact set of degree at least two for other classes of maps.

Let
\begin{equation}\label{d-proper}
g:U_{-1}\to U_0
\end{equation}
be a proper holomorphic map
of some degree $d\ge 2$, where $U_{-1}, U_0$ are Jordan domains.
We study such a map under an extra assumption
that it is backward invariant, i.e., all further (following the first one $U_{-1}=g^{-1}(U_0)$) pullbacks $g^{-i}(U_0)$, $i=2,3,...$, of $U_0$ also exist and share the same critical set $C_g$ (BI map, see Definition \ref{d-bi}).

If $\overline{U_{-1}}\subset U_0$, the map (\ref{d-proper}) is polynomial-like and is always BI because all pullbacks exist
(see Section \ref{ss-ex-1} and Example \ref{ex-1} there for details).
On the other hand, there are two interesting classes of BI maps which are not, in general, PL maps, see Example \ref{ex-2-0} for the first
class and
Examples \ref{ex-3-minus}-\ref{ex-3} for the second one.
If, e.g., in (\ref{d-proper}), $U_{-1}\subset U_0$, but $\overline{U_{-1}}\cap\partial U_0\neq\emptyset$, and the postcritical set is a subset of $U_{-1}$, the map is
BI but not PL.
On the other hand, another class of BI maps appears naturally in connection with the renormalization.
Namely, a BI map can be associated to any simple renormalization of a quadratic polynomial, see Example \ref{ex-3-minus}.
If, moreover, a quadratic polynomial is infinitely renormalizable and robust but with no unbranched complex bounds, in the sense of McMullen \cite{mcm}, a rescaling subsequence of BI maps along renormalizations converges to a limit BI map,
see Example \ref{ex-3} and Section \ref{s-examples}.
In fact, the latter examples
inspired the present work.

The set $C_g:=\{z\in U_{-1}: g'(z)=0\}$
of critical points of the map (\ref{d-proper}) is not empty
as $U_{-1}, U_0$ are simply connected and the degree $d\ge 2$.
Let
$$P_g=\overline{\cup_{n>0} g^n(C_g)}$$
be the postcritical set of this map whenever it is well-defined.
\begin{definition}\label{d-bi}
The map (\ref{d-proper}) is called backward invariant (BI map) if
there exist a sequence
$$U_0, U_{-1},...,U_{-n},...$$
of Jordan domains and a holomorphic extension
$g:\cup_{i=1}^\infty U_{-i}\to \cup_{i=0}^\infty U_{-i}$ of $g: U_{-1}\to U_0$ such that
the following hold, for every $i=0,1,...$:
\begin{itemize}
\item $g(U_{-i-1})=U_{-i}$ and $g:U_{-i-1}\to U_{-i}$ is a proper map of degree $d$,
\item the critical set
$C_g\subset U_{-i-1}$
and the postcritical set $P_g\subset U_{-i}$.
Moreover, we assume that if $P_g$ is a singleton $\{c\}$ then $g^{-1}(c)\neq \{c\}$.

\end{itemize}
\end{definition}
Notes: (1) it follows that all maps $g: U_{-i-1}\to U_{-i}$ share the same set $C_g$ of critical points,
(2) the second condition implies that, for every $n>0$ and $i\ge 0$, $g^n(C_g)\subset g^n(U_{-i-n})=U_{-i}$, in particular,
the postcritical set $P_g$ is well-defined.

In the sequel, slightly abusing notation,
\begin{equation}\label{eq-bi}
g: U'\to U, \mbox{ where } U'=\cup_{i=1}^\infty U_{-i}, U=\cup_{i=0}^\infty U_{-i},
\end{equation}
always refers to the BI map (unless stated otherwise).

We call a BI map $g: U'\to U$ {\it trivial} if there is an attracting fixed point $a\in \cap_{i\ge 0} U_{-i}$ such that
$g^n(z)\to a$ for all $z\in U$.




For every $i\ge 0$ and $n\ge 1$, the map $g^n: U_{-i-n}\to U_{-i}$ is a proper map of degree $d^n$.
By $g_i^{-n}: U_{-i}\to U_{-i-n}$ we denote the inverse (multi-valued) map.
E.g., $g_0^{-n}(U_0)=U_{-n}$.

\begin{remark}\label{r-nocover}
We do not assume that $g: U'\to U$ is a proper map (of degree $d$).
But if it is so (as in Example \ref{ex-2-0}), then writing $g^{-n}(x)$ (i.e. dropping $i$ in $g_i^{-n}(x)$) is not ambiguous and this would allow us to prove a bit more, see Remark \ref{r-degree} and the Complement to Corollary \ref{c-uniq}. However, there is no reason to believe that
$g:U'\to U$ is a proper map in Examples \ref{ex-3-minus}-\ref{ex-3}.
\end{remark}

The following maps (i)-(iii) which are associated to the BI map (\ref{eq-bi}) are BI by themselves: (i)
the restricted map $g: \cup_{n=1}^\infty g_0^{-n}(\Omega)\to \cup_{n=0}^\infty g_0^{-n}(\Omega)$,
for every Jordan domain $\Omega$ such that
$P_g\subset\Omega\subset U_0$, (ii)
the shifted map $g: \cup_{i=i_0+1}^\infty U_{-i}\to \cup_{i=i_0}^\infty U_{-i}$, for every $i_0>0$,
and (iii)
the iterated map $g^n: \cup_{i=1}^\infty U_{-ni}\to \cup_{i=0}^\infty U_{-ni}$, for every $n>1$.
All these maps share the same  postcritical set $P_g$.



\subsection{Polynomial-like maps}\label{ss-ex-1}

\

Recall that a {\it polynomial-like (PL)} map is a triple $(V', V, g)$ where $V', V$ are topological disks in the plane
such that
$\overline{V'}\subset V$ and $g: V'\to V$ is a proper holomorphic map of some degree $d\ge 2$ \cite{DH}.
The non-escaping set $K_g:=\{z: g^n(z)\in U_{-1}, n=0,1,2,...\}$
is called the {\it filled Julia set} of the PL map $g$.

It is easy to see that $K_g$ is a {\it compact} set, it is {\it full},
i.e., its complement $\C\setminus K_g$ is connected,
and is {\it completely invariant}, i.e., $g^{-1}(K_g)=K_g$.

The compact $K_g$ is connected if and only if $C_g\subset K_g$ where $C_g=\{z\in V': g'(z)=0\}$. Equivalently, $P_g\subset V$.

By the Straightening Theorem \cite{DH}, a PL map $g: V'\to V$ is hybrid equivalent (i.e.,
conjugate in some neighborhood of $K_g$ by a quasiconformal homeomorphism which is conformal a.e. on $K_g$) to
a polynomial of degree $d$. Since the polynomial dynamics is well-understood, see e.g., \cite{CG}, this gives a complete description of the dynamics on $K_g$. In particular, the set of repelling periodic points of the map $g: K_g\to K_g$ is dense in $\partial K_g$, for every $n\in\N$, the map $g^n:K_g\to K_g$ has $d^n$ fixed points (counting multiplicity), every component of the interior of $K_g$ is (pre)periodic \cite{Su}, etc. Moreover, the set of limit points of $\{g^{-n}(x)\}_{n\ge 0}$ coincide with $\partial K_g$, for all $x\in K_g$ with the only exception if $g^{-1}(x)=\{x\}$ (in the latter case, $C_g=\{x\}$ and is completely invariant).

A nice feature of a PL map $g: V'\to V$ is its stability: a small perturbation $(\tilde V, \tilde g)$ of the pair $(V, g)$ is again a PL map $\tilde g: \tilde g^{-1}(\tilde V)\to \tilde V$.

Every PL map is, essentially, a BI map:
\begin{exam}\label{ex-1}
Given a PL map $g: V'\to V$ with $P_g\subset V'$,
choose any Jordan domain $U_0\subset V$
which is slightly smaller than $V$ (if $V$ is Jordan itself, one can take $U_0=V$). Then $U_{-i}:=g^{-i}(U_0)$, $i=0,1,2,...$, is a well-defined (decreasing in this case) sequence of Jordan domains such that $P_g\subset U_{-i}$ and $g: U_{-i-1}\to U_{-i}$
is proper degree $d$, for each $i$. Assume additionally that $C_g$ is not a completely invariant singleton
(equivalently, $g: V'\to V$ is not hybrid equivalent to the power map $z\mapsto z^d$).
Then $g:\cup_{i\ge 1}U_{-i}\to \cup_{i\ge 0}U_{-i}$ is a BI map.
\end{exam}

\

See Section \ref{ss-exa} for examples of BI maps which are not polynomial-like.

\

Example \ref{ex-1} makes natural the following
\begin{definition}\label{d-bi-pl}
We say that a BI map $g: U'\to U$ has a polynomial-like restriction (around $P_g$) if
there exists a topological disk $V\subset U_0$ as follows: if $V':=g_0^{-1}(V)$ then
$P_g\subset V'\subset\overline{V'}\subset V$
and
$g:V'\to V$ is a polynomial-like map of degree $d$.
\end{definition}
In this case, the filled Julia set $K$ of the PL map $g:V'\to V$ is a full completely invariant compact containing $P_g$.

\subsection{Main result}
Example \ref{ex-1} provides us a necessary condition for a BI map to have a PL restriction: the existence of a full completely invariant compact set that contains all critical points. We prove that this necessary condition is also sufficient:

\begin{theorem}\label{thm-m}
Suppose that $g: U_{-1}\to U_0$ is a proper holomorphic map of some degree $d\ge 2$
which extends to a BI map $g: U_{-i-1}\to U_{-i}$, $i=1,2,...$.
Let $K\subset \cap_{i\ge 0} U_{-i}$ be a compact set as follows:
\begin{enumerate}
\item [(t1)] $K$ is completely invariant,
\item [(t2)] $K$ is a full compact and $P_g\subset K$.
\end{enumerate}
Then $g$ has a polynomial-like restriction
and $K$ is its filled Julia set. In particular, $K$ is a continuum,
\end{theorem}
\begin{remark}\label{r-mcmdeep}
The first result of this kind seems to appear in \cite[Ch.5.5]{mcm} where a proper analytic map $g: U\to V$ of degree $d\ge 2$ between disks is considered such that the postcritical set $P_g\subset U$.
Assume that $g:U\to V$ has no an attracting fixed point.
There is a constant $M_d>0$ such that if $P_g$ lies in the bounded component of $\C\setminus A$ for some annulus $A\subset U$ of modulus $\mod(A)\ge M_d$, then
a polynomial-like restriction around the postcritical set $P_g$ exists.
For the proof, the author shows that if the space between $P_g$ and $\partial V$ goes to infinity, then $g$ tends to a polynomial of degree $d$. See \cite{mcm} for details. In particular, a completely invariant full compact ($=$filled Julia set of this PL restriction) always exists in this situation of big space around $P_g$. However, such a compact does not exit in general (e.g. in the set up of Example \ref{ex-2-0} below). Theorem \ref{thm-m} gives conditions for the existence of a PL restriction without the "big space" assumption.

See also Remark \ref{r-folk} below.
\end{remark}
\begin{remark}
In \cite{cs},
the situation is considered of an analytic map which leaves a compact set forward
invariant (rather than completely invariant), and a "pruned
polynomial-like" extension of the map is constructed. The authors apply this to study the analytic structure of conjugacy classes in spaces
of real analytic maps of an interval.
\end{remark}
In \cite{L-lf} we use Theorem \ref{thm-m} to study holomorphic/meromorphic invariant line fields of BI maps,
in connection with Example \ref{ex-3} and the main result of \cite{mcm}.
For another application of methods and results of the present paper, see \cite{L-cb}.


Here, we derive the following corollary and its complement.
Given a BI map $g: U'\to U$ we define, for each $n\in\N$, the set
\begin{equation}\label{eq-Rn}
R_n=\{z| z\in U_{-n}\cap U_0, g^n(z)=z, |(g^n)'(z)|>1\}.
\end{equation}
Note that $R_n$ is finite in any compact subset of $U_0$, by the Uniqueness theorem.
Let
\begin{equation}\label{eq-R}
R'=\{z| \exists z_i\to z, z_i\in R_{n_i}, n_i\to\infty\}.
\end{equation}

\begin{coro}\label{c-uniq}
For a non-trivial BI map $g: \cup_{i\ge 1}U_{-i}\to \cup_{i\ge 0}U_{-i}$, the following alternative holds:

(a) $g$ does not have a completely invariant full compact $K$ such that
$P_g\subset K\subset \cap_{i\ge 0}U_{-i}$;
then the set $R'$ contains a (non trivial) continuum which has common points with $U_0$ as well as with its boundary $\partial U_0$.

(b) $g$ has a completely invariant full compact $K$ such that
$P_g\subset K\subset \cap_{i\ge 0}U_{-i}$;
then there exists a neighborhood $W$ of $K$, such that $g:g_0^{-1}(W)\to W$ is a PL restriction and $K$ its filled Julia set, moreover,
\begin{enumerate}
\item $R'\cap K=\partial K$,
\item $R_n\cap (W\setminus K)=\emptyset$, for every $n\in\N$,
\item $R'\setminus K\subset\partial U_0$.
\end{enumerate}

\end{coro}

\begin{remark}\label{r-degree}
Assume that a BI map $g: U'\to U$ is a degree $d$ proper map. Then $\cup_{n\ge 1}R_n\subset R'$,
in particular, in case (b), $g$ has no repelling periodic orbits outside of $K$.
Indeed, $g: U_{-i-1}\to U_{-i}$ has degree $d$ and if also $g: \cup_{i\ge 1}U_{-i}\to \cup_{i\ge 0}U_{-i}$ has degree $d$ then $g^{-1}(U_{-i})=U_{-i-1}$. This implies that $g_0^{-n}(x)=g^{-n}(x)$ for all $x\in U_0$ and all $n\in\N$. If now $z\in U_{-n}\cap U_0$ is such that $g^n(z)=z$, then $g^{ni}(z)=z$, hence, $z\in U_{-ni}$. Therefore, letting $z=z_i$ and $n_i=ni$ in (\ref{eq-R}) we see that $z\in R'$.
\end{remark}
This remark is relevant for, e.g., a class of BI maps of Example \ref{ex-2-0}. Let us state this explicitly as the following

{\bf Complement to Corollary \ref{c-uniq}.}
{\it Let $V', V$ be simply connected domains in $\C$ such that $V'\subset V$, $V'\neq V$ and
$g: V'\to V$ is a proper holomorphic map of some degree $d\ge 2$
and the postcritical set $P_g\subset V'$. Let $R'$ be the closure of all repelling cycles of $g: V'\to V$.
Then the alternative (a)-(b) holds. Moreover, in case (b) of the alternative, $R'\subset K$, and
for any $x\in V'\setminus K$, either $g^n(x)\in V\setminus V'$ for some $n\ge 0$, or the forward orbit
$O(x)$ of $x$ is well-defined
but
$$\omega(x)\subset\partial V,$$
where $\omega(x)$ is the limit set of $O(x)$.}


Next remark is, essentially, a folklore:
\begin{remark}\label{r-folk}
For a proper map (\ref{d-proper}), conditions (t1)-(t2) alone are not enough for the conclusion of Theorem \ref{thm-m} to hold.
Indeed, let's assume, along with (t1)-(t2), that $g: U_{-1}\to U_0$ has an attracting or parabolic cycle
such that its basin of attraction intersects $U_{-1}\setminus K$. Then it is easy to see that the conclusion of Theorem \ref{thm-m} breaks down.
Nevertheless, if, additionally, to (t1)-(t2),
$K$ is connected and attracting/parabolic cycles as above do not exist, then the conclusion of Theorem \ref{thm-m} does hold
(see \cite{bbb} for details, see also Case (B) in the proof of Proposition \ref{p-bi} below).
The proof of this is heavily based on the fact that $K$ is a connected full compact set. This allows us to conjugate the map
$g: U_{-1}\setminus K\to U_0\setminus K$ (by uniformazing $\C\setminus K$) to a real analytic covering map on an outer neighborhood of the unit circle $S^1$ with no critical points, extend it through $S^1$ by the reflection principle and then apply a non-trivial Mane's theorem
which says that
a $C^2$ covering map of the circle without critical points is expanding away from its attracting and parabolic cycles.


\end{remark}

The proof of Theorem \ref{thm-m} is a mixture of dynamical and geometric arguments.
The main step
constitutes the following
\begin{theorem}\label{thm-mainstep}
Under the conditions of Theorem \ref{thm-m}, the set $R'$ (see \ref{eq-Rn}-\ref{eq-R}) has no points in $U_0\setminus K$.
\end{theorem}
The proof of Theorem \ref{thm-mainstep} in the case of connected $K$ is rather simple because we can still use the conjugation of the map $g: U_{-1}\setminus K\to U_0\setminus K$ to a circle map (although the proof is more elementary than the one outlined in the above Remark \ref{r-folk}  as Mane's theorem is not needed), see Sect. \ref{ss-k-conn}.
Most part of the paper occupies the proof of Theorem \ref{thm-mainstep} in the case of disconnected $K$ along with Appendices A-C
where we collect auxiliary results which are used in the proof.

\subsection{Generalized BI maps}\label{ss-gen}

Note that {\it a priori} the completely invariant compact $K$ of Theorem \ref{thm-m} does not have to be connected, although {\it a posteriori} it 
does.

Let us mention the following generalization of Theorem \ref{thm-m} where $K$ is never connected.
Let, in Definition \ref{d-bi}, $U_0$ be still a Jordan domain while $U_{-1}$ be an open set consisting of finitely many Jordan components $U_{-1}^k$, $k=1,...,m$,
for some $m\ge 2$.
Let $g: U_{-1}\to U_0$ be a proper map of degree $d\ge 2$ meaning that each $g: U_{-1}^k\to U_0$ is a proper map of some degree $d_k\ge 1$ such that $d=\sum_k d_k$.
We say that $g$ extends to a {\it generalized BI map} if
there exist a sequence
$U_{-1},...,U_{-i},...$
of bounded open sets each of which consists of finitely many Jordan components,
and a holomorphic extension
$g:\cup_{i=1}^\infty U_{-i}\to \cup_{i=0}^\infty U_{-i}$ of $g: U_{-1}\to U_0$ such that, for every $i=0,1,...$:
(1) $g(U_{-i-1})=U_{-i}$ and $g: U_{-i-1}\to U_{-i}$
is a proper holomorphic map of degree $d$,
(2) either
the set $C_g$ of the critical points of $g$ is empty, or
$C_g\subset U_{-i-1}$ and
$P_g:=\overline{\cup_{n>0} g^n(C_g)}\subset U_{-i}$,
for all $i\ge 0$.

Recall that a
{\it generalized PL map} \cite{LM} $G: V'\to V$ is a holomorphic map between open sets $V',V$, where $V$ is simply connected, $\overline{V'}\subset V$, and $V'$ consists of finitely many simply connected components $V_k'$, $k=1,...,M$, with pairwise disjoint closures, such that each restriction $G: V_k'\to V$ is a proper map of some degree $D_k\ge 1$.
Its set of non-escaping points (filled Julia set) $K_G=\{z: G^n(z)\in V', n=0,1,...\}$ is a completely invariant full compact.
(Note that if $V'$ is connected, $G: V'\to V$ is the classical (Douady-Hubbard) PL map.)
The Straightening theorem also holds: $G: V'\to V$ is conjugated, in a neighborhood of $K_G$, to a polynomial of degree
$D=\sum_{k=1}^M D_k$ (see e.g. \cite{LS} for a proof of a similar claim).
If $M\ge 2$, then $K_G$ is never connected.

\begin{theorem}\label{thm-gen-bi}
Suppose that $g: \cup_{i\ge 1} U_{-i}\to \cup_{i\ge 0} U_{-i}$ is a generalized BI map.
Let $K\subset\cap_{i\ge 0}U_{-i}$ be a full completely invariant compact set such that
either $g$ has no critical points or $P_g\subset K$.
Then there exist a simply connected domain $V_0$, $\overline{V_0}\subset U_0$, and $m\ge 1$ such that $\overline{V_{-m}}\subset V_0$, where $V_{-m}=g_0^{-m}(V_0)$,
and $K$ is the filled Julia set of the generalized PL map $g^m: U_{-m}\to U_0$ of degree $d^m$.
Furthermore, there exists a neighborhood $V\subset U_0$ of $K$ consisting of finitely many pairwise disjoint with their closures
simply connected domains such that $\overline{g_0^{-1}(V)}\subset V$ and $g: g^{-1}(V)\to V$ is a proper map of degree $d$.

\end{theorem}

\subsection{Description of the paper}
Section \ref{ss-exa} contains examples of BI maps which are not polynomial-like.

In Section \ref{s-bi} some general properties of BI maps are collected. Primarily, we prove in Proposition \ref{p-bi} that any BI map which is not trivial (i.e., not a basin of attraction of an attracting fixed point) either has a lot of repelling cycles close to the boundary and does not admit a PL restriction, or has a PL restriction, with filled Julia set containing almost all repelling cycles.

Theorem \ref{thm-mainstep} is proved in Section \ref{s-mainstep-proof}. The proof in the case of disconnected compact $K$ is based on general facts and their corollaries from Appendices A-C.

Short Sections \ref{s-thm-m}-\ref{s-c-uniq}-\ref{s-thm-gen-bi} contain proofs of Theorem \ref{thm-m} (as a simple consequence of Theorem \ref{thm-mainstep}), Corollary \ref{c-uniq} along with its Complement,
and Theorem \ref{thm-gen-bi}, respectively.

Section \ref{s-examples} contains a detailed exposition of Example \ref{ex-3-minus} and (mainly) Example \ref{ex-3}.

Appendix A. We prove following closely \cite{MP}, \cite{P} that a locally analytic self-map of a compact set $K$ of degree $d\ge 2$
has topological entropy $\ge \log d$. Coupled with the Variational Principle and the formula for the Hausdorff dimension of invariant measures, this implies that the logarithmic capacity of the completely invariant compact set is positive.

Appendix B. Given a (disconnected) compact $X\subset\C$ of positive logarithmic capacity and a universal covering $\pi:\D\to\C\setminus X$,
we recall a classical result about the existence and properties of an infinite Blaschke product $B:\D\to\D$ such that its projection is Green's function of $\D\setminus X$. In Corollary \ref{co-basic} we use Lusin-Privalov's construction to build a nice simply connected domain $W\subset\D$, in particular,
the length of $\partial W\cap\partial\D$ is positive, $\pi(W)$ is contained in a given neighborhood of $X$ and $\pi(w)\to X$ as
$w\to\partial W\cap\partial\D$ in $W$.

Appendix C. We consider a univalent map $f:\D\to\D$ and show, based on an inequality proved in \cite{nehari}, \cite{pomm-vas}, \cite{sol} (see also Lemma \ref{l-extr}) that, if $E, f(E)\subset \partial \D$, then the length of $E$ is small provided $|f'(0)|$ is small
and $f(0)$ is away from the unit circle.

{\bf Acknowledgments.} I would like to thank Feliks Przytycki for several fruitful discussions concerning Theorem \ref{thm-entropy},
and Misha Sodin, Sebastian van Strien and Vladlen Timorin for helpful remarks and references.
I thank also the referee for helpful comments.

\section{Examples of BI maps}\label{ss-exa}

The next example generalises the previous Example \ref{ex-1}. We obtain BI maps by relaxing some conditions in the definition of polynomial-like maps and in the choice of $U_0$.
\begin{exam}\label{ex-2-0}
Let $V', V$ be simply connected (perhaps, unbounded) domains in $\C$ such that $V'\subset V$, $V'\neq V$ and
$g: V'\to V$ is a proper holomorphic map of some degree $d\ge 2$. We don't assume that $V'$  is compactly contained in $V$.
Such a map must have a non-empty set of critical points $C_g=\{x: g'(x)=0\}$.
Assume that the postcritical set $P_g=\cup_{n>0}g^n(C_g)$
is well-defined and is a compact subset of $V'$, and if $P_g$ is a singleton $\{c\}$ then $g^{-1}(c)\neq \{c\}$.
Let $U_0\subset V$ be an arbitrary Jordan domain such that $P_g\subset U_0$. Define
$U_{-i}=g^{-i}(U)$, $n=0,1,2,...$.
Then $g:\cup_{i\ge 1}U_{-i}\to \cup_{i\ge 0}U_{-i}$ is a BI map.




{\it Sub-example:} Let $V', V$ and $U_0$ be as above, moreover, $V', V$ are bounded and $\overline{V'}\subset V$ (i.e., $g: V'\to V$ is PL).
Then it is not difficult to see that the obtained BI map has a PL restriction if and only if $K_g\subset U_0$, where $K_g$ is the filled Julia set
of the PL map $g: V'\to V$
(the PL restriction is then $g: g^{-n}(V')\to g^{-n}(V)$, for some $n>0$ such that $g^{-n}(V)\subset U_0$).

\end{exam}
The last two examples of BI maps are related to renormalizable quadratic polynomials, {\it see Section \ref{s-examples} for details}.
\begin{exam}\label{ex-3-minus}
Let $f(z)=z^2+c$ be a quadratic polynomial. Assume that it admits a simple renormalization of a period $n>1$ (see \cite{mcm} for definitions).
Let $J_n$ be the filled Julia set of a PL restriction of $g^n$ with a single critical point at $0$, and let $P_n$ be the postcritical set of this PL map. Note that $P_n=J_n\cap P_f$ where $P_f$ is the postcritical set of $f:\C\to\C$.
Let $\gamma_n$ be a
Jordan curve in $M:=\hat\C\setminus P_f$ which is homotopic in $M$ to a loop that separates $J_n$ from $P_f\setminus J_n$.
Let $U_0$ be a Jordan domain bounded by $\gamma_n$.
Pulling back $U_0$ along the renormalization as in \cite{mcm}, we obtain a domain $U_{-1}\supset P_n\cup\{0\}$ and degree two proper restriction map
$f^n: U_{-1}\to U_0$ with the critical set $C_g=\{0\}$ and the postcritical set $P_n$. Similarly, we pull back $U_{-1}$ getting $U_{-2}$
such that $f^n: U_{-2}\to U_{-1}$ is degree $2$ proper map and $\{0\}\cup P_n\subset U_{-2}$, and so on.
As a result, we obtain a BI map $f^n: U_{-i-1}\to U_{-i}$, $i=0,1,...$, with the postcritical set $P_n$
\end{exam}


\begin{exam}\label{ex-3}
Let $f(z)=z^2+c$ be infinitely renormalizable. Assume that $f$ is robust and has no unbranched complex bounds
(see \cite{mcm} and Section \ref{s-examples}). Then a limit BI map naturally appears as follows. First, at each period $n$ of simple renormalization we have a BI map obtained as in the previous example, where we start with a geodesic $\gamma_n$ of the Riemann surface $M$ which is homotopic to a loop in $M$ around $J_n$. Then we rescale the sequence of obtained BI maps to the same scale and pass to a limit, see Section \ref{s-examples}.
\end{exam}



\section{On BI maps}\label{s-bi}
Let $g:\cup_{i\ge 0}U_{-i-1}\to\cup_{i\ge 0}U_{-i}$ be a BI map. Recall that $g_i^{-1}$ is the multi-valued inverse to the proper map $g:U_{-i-1}\to U_i$, $i=0,1,...$.
\begin{lemma}\label{l-compl-inv}
Let $Y\subset \cap_{i\ge 0}U_{-i}$ be a compact set. The following conditions are equivalent:
\begin{enumerate}
\item $g^{-1}(Y)=Y$, i.e., $Y$ is completely invariant,
\item $g_i^{-1}(Y)=Y$ for all $i\ge 0$,
\item $g_i^{-1}(Y)=Y$ for some $i\ge 0$ and $g_j^{-1}(x)=g_i^{-1}(x)$ for any $j$ and any $x\in Y$,
\item $g: Y\to Y$ has degree $d$ (counting multiplicities).
\end{enumerate}
Moreover, if any of (1)-(4) holds, then, given $n>0$, for every neighborhood of $W$ of $Y$ there exists another neighborhood $W'$ of $Y$
as follows: for every $j=1,...,n$, two (multi-valued) maps $g_0^{-j}=g_{j-1}^{-1}\circ...\circ g_0^{-1}$ and $(g_0^{-1})^j=g_0^{-1}\circ...\circ g_0^{-1}$ ($j$ times) are well-defined on $W'$,
$g_0^{-j}=(g_0^{-1})^j$ on $W'$ (as multi-valued maps), and
$g_0^{-j}(W')=(g_0^{-1})^j(W')\subset W$.
\end{lemma}
\begin{proof} (1)$\Rightarrow$ (2): given $i$, (1) implies that $g_i^{-1}(Y)\subset Y$ and since $g_i=g$ on $U_{-i}\supset Y$, $g_i(Y)=Y$, hence,
$Y\subset g_i^{-1}(Y)$. Thus $Y\subset g_i^{-1}(Y)\subset Y$, i.e, (2). (2)$\Rightarrow$ (1): obvious.
(2)$\Rightarrow$ (3): given $i$, $g_i^{-1}(Y)=Y$ by (2). Let $j\neq i$. As $g: U_{-k-1}\to U_{-k}$ is a branched $d$-covering map for all $k$, for every $x\in Y\setminus g(C_g)$,
$g_i^{-1}(x)$ consists of $d$ pairwise distinct points which are all in $Y\subset U_{-j}$.
Assume $y\in g_j^{-1}(x)\setminus g^{-1}(x)$, then, using (2), $y\in Y\subset U_{-i}\cap U_j$, hence, $y\in g_i^{-1}(x)$,
a contradiction. Thus $g_i^{-1}(x)=g_j^{-1}(x)$ for all $x\in Y\setminus g(C_g)$, hence, by continuity, for all $x\in Y$
(3)$\Rightarrow$ (2): obvious. Thus, (1)-(3) are equivalent.
(3)$\Rightarrow$ (4): obvious. (4)$\Rightarrow$ (3): as $g=g_i$ on $Y$ and $g_i:U_{-i-1}\to U_{-i}$ has degree $d$, the implication follows.

Let's verify the last claim. Let $W$ be a neighborhood of $Y$. Fix $j$. By item (3), $g_0^{-j}(x_0)=(g_0^{-1})^j(x_0)$ for $x_0\in Y$ and all preimages are also in $Y$,
hence, this holds for all $x$ in a neighborhood $W_{x_0}$ of $x_0$ (with all preimages of $x$ in neighborhoods of corresponding preimages of $x_0$ and in $W$).
By compactness of $Y$, the union of finitely many $W_{x_0}$ form a neighborhood $W'_j$ of $Y$. Then $W'=\cap_{j=0}^{n-1} W'_j$ is as required.

\end{proof}

Next, we are going to describe $3$ alternatives for a BI map.
Let $g: U'\to U$ be a BI map, where $U'=\cup_{i=1}^\infty U_{-i}$, $U=\cup_{i=0}^\infty U_{-i}$.
Let us fix a Jordan curve $\Gamma_0$ in $U_0$ which encloses the postcritical set $P_g$.
Let $V_0$ be the Jordan domain bounded by $\Gamma_0$, so that $\overline{V_0}\subset U_0$.

Recall that $g_i^{-n}:U_{-i}\to U_{-i-n}$ is the inverse (multi-valued) function to $g^n: U_{-i-n}\to U_{-i}$.

Let, for each $n>0$, $V_{-n}=g_0^{-n}(V_0)$ (which is compactly contained in $U_{-n}$) and $\Gamma_{n}=\partial V_{-n}$
(which is also a Jordan curve).

A priori, there are $3$ possibilities:

(A) $\Gamma_n\cap\Gamma_0\neq\emptyset$ for all $n$.

(B) For some $n>0$, $\overline V_{-n}\subset V_0$.

(C) For some $n>0$, $\overline V_0\subset V_{-n}$.


\begin{prop}\label{p-bi}
Assume that (A) holds. Then, for every $\eps>0$ and every $n$ big enough, the map $g^n: U_{-n}\to U_0$ has a repelling fixed point
in $\eps$-neighborhood of $\Gamma_0$.

Assume that (B) holds. Then $g^n: V_{-n}\to V_0$ is a PL map of degree $d^n$. Assume additionally that
there is a full compact $K\subset \cap_{i\ge 0}U_{-i}$
such that $g^{-1}(K)=K$ and $K\subset V_0$.
Then $P_g\subset K$ and there exists a polynomial-like restriction
$g: V'\to V$ where $V$ is a simply connected neighborhood of $K$ and $K$ is the filled Julia set of $g: V'\to V$.

Assume (C) holds.
Then
there is an attracting fixed point
$a\in \cap_{i\ge 0}g^{-i}(V_0)$ such that
$g^n(z)\to a$ for all $z\in \cup_{i\ge 0}g^{-i}(V_0)$.
Moreover, $g$ has no a completely invariant compact.
Finally, if the alternative (C) holds for a sequence of $\Gamma_0$'s that tends to $\partial U$, then the BI map $g: U'\to U$ is trivial: there is an attracting fixed point $a\in \cap_{i\ge 0} U_{-i}$ such that
$g^n(z)\to a$ for all $z\in U$.
\end{prop}

\begin{proof}

{\bf (A)}. In this case, there exist a sequence of indices $n_i\to\infty$ and a sequences of points $x_i\in\Gamma_{n_i}\cap\Gamma_0$ as follows:
$x_i\to b'$, $y_i:=g^{n_i}(x_i)\in \Gamma_0$, $y_i\to b$ and $(g^{n_i})'(x_i)\to L$ where $L\in\C\cup\{\infty\}$.
Fix a small disk $B$ centered at $b$ which is contained in $U_0$ and disjoint with $P_g$, and, for each $i$, let
$\phi_i$ be a branch of $g_0^{-n_i}$ which is defined in $B$ and maps $y_i$ to $x_i$. Then $(\phi_i)'(y_i)\to 1/L$.

Thus $\phi_i:B\to\C$, $i=1,2,...$, is a sequence of univalent maps and is a normal family because it avoids taking values in the set
$P_g\cup g^{-1}(P_g)$, which must contain at least $2$ points by Definition \ref{d-bi}.
By this reason, the case $L=0$ is impossible.

Consider the case $L\neq \infty$, that is, $0<|L|<\infty$.
Replacing $(n_i)$ by a subsequence, in this case $(\phi_i)$ tends uniformly on compacts in $B$ to a non-constant, hence, univalent function $\nu$ which maps $B$ to a neighborhood $B'$ of $b'$. As $\psi_i(B)\subset g_0^{-n_i}(U_0)=U_{-n_i}$, then $B'\subset U_{-n_i}$ for all $i$ big enough.
Shrinking slightly $B'$, we get:

{\it Claim.} The sequence $g^{n_i}: B'\to B$ is well-defined and converges uniformly to $\nu^{-1}$.

Let $m_i=n_{n+1}-n_{n_i}$ Passing again to a subsequence of $(n_i)$,
one can assume that $m_i\to\infty$. Write $g^{n_{i+1}}=g^{m_i}\circ g^{n_i}$, on $B'$. Here,
$g^{n_i}(B')\subset g^{n_i}(U_{-n_{i+1}})\subset U_{-m_i}$.

Shrinking again $B'$ if necessary, by the Claim, $g^{m_i}\to \id$, uniformly on $\nu^{-1}(B')$ where $\nu^{-1}(B')\subset U_{-m_i}$.
Concluding this argument (and shrinking $B$ to a smaller disk around $b$), we obtain that, for some branches $I_i:=g_0^{-m_i}$ on $B$,
$I_i\to \id$ uniformly on $B\subset U_0$.

Let now $D$ be any simply connected subdomain of $U_0\setminus P_g$ containing $B$. As all branches of $g_0^{-m_i}$
are well-defined in $D$ and form a normal family, passing to a subsequence, we get that $I_i$ (more precisely, their extensions from $B$ to $D$) converge uniformly on compacts in $D$ to $\id$. As each $I_i$ is an inverse to $g^{m_i}$, it follows that, given a compactly contained
in $D$ domain $F\subset D$, for all $i$ big enough, we have:
$F\subset U_{-m_i}$ and $g^{m_i}\to \id$ uniformly on $F$.

We apply this to two domains $D_1, D_2$ which together cover the curve $\Gamma_0$, and to their compactly contained subdomains $F_1\subset D_1$, $F_2\subset D_2$, such that $F_1\cup F_2\supset \Gamma_0$. As a result, we obtain that $\Gamma_0\subset U_{-m_i}$ for all $i$ big enough and
$g^{m_i}\to \id$ uniformly on $\Gamma_0$. This is a contradiction by the following reason:
on the one hand, as we have just shown, $\Gamma_0\subset U_{-m_i}$, also $\Gamma_0=\partial V_0$ and $g^{m_i}(U_{-m_i})=U_0$, hence, the sequence $g^{m_i}$ is well-defined on $\overline{V_0}$ and converges uniformly to the identity map $\id$. Then $(g^{m_i})'\to 1$ uniformly on compacts in $V_0$.
On the other hand, $(g^{m_i})'(c)=0$ for each $c\in C_g$
This contradiction excludes the case $0<|L|<\infty$.

Thus we must have that
$(g^{n_i})'(x_i)\to \infty$, i.e., $(\psi_i)'(y_i)\to 0$. Let us show that then $g^{n_i}$ must have a repelling fixed point close to the curve $\Gamma_0$.
Indeed, let $D\subset U_0\setminus P_g$ be a simply connected domain that contains points $b'=\lim x_i$ and $b=\lim y_i$.
Then, by the normality of the family of extensions of $\phi_i$ from $B$ to $D$, a subsequence of $\phi_i$ converges uniformly on compacts
in $D$ to a point $A\in D$. In particular, for a small disk $B_A$ centered at $A$, and for some branch $\phi_i$ of $g^{-n_i}$,
$\overline{\phi_i(B_A)}\subset B_A$. Therefore, $\phi_i$ has an attracting fixed point $a\in B_A$, which is thus a repelling fixed point of $g^{n_i}$.


{\bf (B)}. $\overline V_{-n}\subset V_0$ for some fixed $n>0$. Then $G:=g^n: V_{-n}\to V_0$ is a polynomial-like map (of degree $d^n$).
Let $K_G$ be the filled Julia set of $G$.
Since $g^{-1}(K)=K$, $K\subset V_0$, and using Lemma \ref{l-compl-inv}, $G^{-1}(K)=g_0^{-n}(K)=K$ and
$K\subset \cap_{j\ge 0} G^{-j}(V_0)=K_G$.
As $G: V_{-n}\to V_0$ is conjugate to a polynomial and $\# K\ge 2$, it follows that $K=K_G$. Moreover, as $P_G=P_g\subset V_{-n}$, hence,
$K=K_G$ is connected and $P_g\subset K$.
It remains to show $g$ itself has a PL restriction, i.e., there exist a simply connected neighborhood $V$ of $K$, such that $V':=g^{-1}(V)$ is compactly contained in $V$.
If $n=1$, there is nothing to prove (set $V=V_0$).

So let $n\ge 2$. We present two proofs: the first one uses heavily that $K$ is a continuum while the second one is "continuum free" and can be easily adapted
to the set up of Theorem \ref{thm-gen-bi} where $K$ is disconnected.

{\it First proof.} The following argument is well-known.
Let $\Psi:\C\setminus K\to \D^*:=\{|w|>1\}$
be the conformal isomorphism such that $\Psi(\infty)=\infty$. Below, $\hat A=\Psi(A\setminus K)$ for a neighborhood $A$ of $K$, and $A_*$ is
the union of $\hat A$ with its symmetric w.r.t. $S^1$ completed by $S^1$. Similar notations are used for maps.
Replace $g: U_{-1}\setminus K \to U_0\setminus K$ by
$\hat g:=\Psi\circ g\circ \Psi^{-1}: \hat U_{-1}\to\hat U_0$.
By Schwartz's reflexion principle, $\hat g$ extends to a holomorphic map $g_*: (U_{-1})_*\to (U_0)_*$.
The latter map has no critical points as $g_*^{-1}(S^1)=S^1$.
Let $\hat G=\Psi\circ g\circ \Psi^{-1}: \hat U_{-m}\to\hat U_0$
and $G_*: (U_{-m})_*\to (U_{0})_*$. As this map is a covering and $\overline{(U_{-m})_*}\subset (U_{0})_*$, then
$G_*$ expands the hyperbolic metric $\rho$ of the domain $(U_{0})_*$, i.e, $D_G(w):=\rho\circ G(w) |G'(w)|/\rho(w) >1$ for $w\in (U_{-m})_*$. Now note that
$G_*=g_*^m$ in some neighborhood of $S^1$. Define a new metric $\tilde \rho=\sum_{i=0}^{m-1}\rho\circ g_*^i |(g_*^i)'|$.
Then
$\tilde D_{g_*}(w):=\frac{\tilde\rho\circ g_*(w) |g_*'(w)|}{\tilde\rho(w)}=1+\frac{\rho(w)}{\tilde\rho(w)}\{D_G(w)-1\}>1$
for all $w$ near $S^1$. By the compactness, $\tilde D_{g_*}(w) \ge\lambda$ for some $\lambda>1$
and all $w$ in a small annulus around $S^1$. It follows that there is a neighborhood $V_*$ of $S^1$ which is symmetric
w.r.t. $S^1$ such that
$g_*^{-1}(V_*)$ is compactly contained in $V_*$. Then we let $V=\Psi^{-1}(V_*\cap \D^*)\cup K$.

{\it Second proof.} Here we don't use that $K$ is a continuum.

As $V_{-n}=g_0^{-n}(V_0)$ and $\overline{V_{-n}}\subset V_0$, then
the iterates $(g_0^{-n})^j$ ($j>0$) of $g_0^{-n}$ are all well-defined on $V_0$ and
$\tilde V_j:=(g_0^{-n})^j(V_0)$, $j=0,1,...$, form a decreasing sequence of Jordan domains (with $\tilde V_0=V_0$) such that
\begin{equation}\label{eq-sec-pr}
K=K_G=\cap_{j>0} \tilde V_j.
\end{equation}
Note that $\tilde V_{j+1}=g_0^{-n}(\tilde V_j)$ and $\overline{\tilde V_{j+1}}\subset \tilde V_j$.

For every $j\ge 0$, let
$\Omega_j:=\cup_{i=0}^{n-1}g_0^{-i}(\tilde V_j)$.

Apply the last claim of Lemma \ref{l-compl-inv} to a neighborhood $W=V_0$ of $K$ and find a neighborhood $W'$ of $K$
such that
\begin{equation}\label{eq-sec-pr-2}
g_0^{-i}=(g_0^{-1})^i \mbox{ for }i=0,1,...,n.
\end{equation}

By (\ref{eq-sec-pr}), there exists $j_*>0$ such that, for every $j\ge j_*$,
$K\subset \Omega_{j}\subset W'$
and $\Omega_j$ shrink to $K$ as $j\to\infty$.

This allows us to write, using (\ref{eq-sec-pr-2}), for every $j\ge j_*$:
$$g_0^{-1}(\Omega_j)=g_0^{-1}(\cup_{i=0}^{n-1}(g_0^{-1})^i(\tilde V_j))=\cup_{i=1}^{n}(g_0^{-1})^i(\tilde V_j)=$$
$$=\Omega_j\cup [(g_0^{-1})^{n}(\tilde V_j)\setminus \tilde V_j)]=
\Omega_j\cup [(g_0)^{-n}(\tilde V_j)\setminus \tilde V_j)]=$$
$$=\Omega_j\cup \tilde V_{j+1}\setminus \tilde V_j\subset\Omega_j.$$
Thus $g_0^{-1}(\Omega_j)\subset \Omega_j$
and
$$g: g_0^{-1}(\Omega_j)\setminus K \to \Omega_j\setminus K$$
is an unbranched covering map, for each $j\ge j_*$.

Let
$\Omega^*=\Omega_{j_*}\setminus K$ and let $\rho_*$ be the hyperbolic metric of the domain $\Omega^*$.
As $g_0^{-1}(\Omega^*)\subset \Omega^*$ and there is no an equality, by the above and the Schwartz lemma,
every local branch of $g_0^{-1}: \Omega^*\to\Omega^*$ is a strict contraction.

Now, fix $j$ large enough, such that $\overline{\Omega_j}\subset\Omega_{j_*}$.

Let $Y$ be the topological hull of $\overline{\Omega_j}$.
As $g_0^{-1}(\overline{\Omega_j})\subset \overline{\Omega_j}$ and $g: U_{-1}\to U_0$ is a proper map between simply connected domains,
$g_0^{-1}(Y)\subset Y$. Hence, if $Y^*=Y\setminus K$, then $g_0^{-1}(Y^*)\subset Y^*$ as well (remind that $K$ is completely invariant).

The set $\Delta:=Y\setminus Int(g_0^{-1}(Y))$ is a compact subset of $\Omega^*$, therefore, there exist $r>0$ and $\kappa\in (0,1)$ as follows:
for every $a\in \Delta$, $B_h(a, 2r)$ is a topological disk and every branch $g_0^{-1}: B_h(a, 2r)\to \Omega_*$ contracts the distance in $\Omega^*$ by a factor at leat $\kappa$
(here and below in the proof $B_h(x, t)$ is the disk centered at $x\in \Omega^*$ and of radius $t>0$ in the hyperbolic metric $\rho_*$ of $\Omega^*$).

Note that $\partial Y\subset\partial\Omega_j\subset \Omega^*$ and $\partial Y^*$ is the disjoint union of $\partial Y$ and $\partial K$.
Let $Y^*_r=Y^*\cup_{a\in \partial Y}\bar{B_h}(a, r)$.

Let $\Gamma_*$ be the outer boundary of $Y^*_r$ (i.e, the boundary of the topological hull of $Y^*_r$).
Note that $\Gamma_*$ is a compact subset of $\Omega^*$.
Let $V_*$ be a (simply connected) domain
bounded by $\Gamma_*$. We claim that $g_0^{-1}(V_*)$ is compactly contained in $V_*$. It's enough to show that
$g_0^{-1}(\Gamma_*)\subset V_*$. To this end, let $z\in\Gamma_*$ and $z'\in g_0^{-1}(z)$.
There exist $a\in \partial Y$, such that $\rho_*(z,a)\le r$, and a branch $\psi$ of $g_0^{-1}$ in a neighborhood of $\bar B_h(a,r)$ such that $\psi(z)=z'$.
If $a':=\psi(a)$, then $\rho_*(a',z')\le k\rho_*(a,z)\le kr<r$.
As $a'\in g_0^{-1}(\partial Y)\subset \Delta$,
it follows that
$z'\in Y^*_r$. Thus $V'_*:=g_0^{-1}(V_*)$ is compactly in $V_*$, i.e., $g: V_*'\to V*$ is the required PL map.

{\bf (C)}. If $\overline V_0\subset V_{-n}$ for some $n>0$, then $\overline{g^n(V_0)}\subset V_0$.
Hence, by the Wolff-Denjoy theorem (e.g., \cite{CG}), the sequence of domains $g^{kn}(V_0)$, $k=1,2,...$, shrink to an attracting fixed point $a\in V_0$ of $g^n$. Then $g^i(\overline{V_0})$, $i\ge n$, tend to the cycle $O\ni a$. As $P_g\subset V_0$ is a closed invariant set, then
 $O\subset P_g\subset V_0$.
Therefore, $O=g^i(O)\subset g^i(V_0)$ for all $i\ge n$ where $g^{nk+j}(V_0)$ shrinks to $g^j(a)$ as $k\to\infty$.
Since $P_g\subset g^i(V_0)$ for all $i\ge n$, it follows that $O=\{a\}$.
If a completely invariant compact $K\subset V$ exists, then
$K=\{a\}$ and $\#f^{-1}(a)=a$. In other words $P_g=\{a\}$ and
$a$ is a fixed critical point of order $d$,
which is excluded by Definition \ref{d-bi}.
Finally, if the alternative (C) holds for a sequence of such $\Gamma_0$'s that tends to $\partial U_0$, then the sequence of corresponding
Jordan domains $V_0$'s bounded by $\Gamma_0$'s exhausts $U_0$, hence, by the above,
there is an attracting fixed point
$a\in \cap_{i\ge 0}g^{-i}(U_0)$ such that
$g^n(z)\to a$ for all $z\in U=\cup_{i\ge i} U_0$, i.e., $g: U'\to U$ is trivial.
\end{proof}

\section{Proof of Theorem
\ref{thm-mainstep}}\label{s-mainstep-proof}

Assume the contrary: $g: U'\to U$ has a completely invariant full compact set $K$ such that $P_g\subset K\subset\cap_{n\ge 0}U_{-n}$, and
there exist a point $a\in U_0\setminus K$ and sequences $n_i\to\infty$ and $a_i\to a$ such that
$a_i\in U_{-n_i}$, $g^{n_i}(a_i)=a_i$ and, for each $i$, $|\lambda_i|>1$ where $\lambda_i=(g^{n_i})'(a_i)$.

For every $i$, as $a_i\in U_0\cap U_{-n_i}$ and $|\lambda_i|>1$, there exists a branch $\psi_{n_i}$ of $g_0^{-n_i}$ which is defined (hence, univalent) in some
neighborhood of $a_i$ such that $\psi(a_i)=a_i$.
As $a\notin P_g$, there exists a disk $Z:=B(a, r)\subset U_0$ such that $\overline{Z}\subset U_0\setminus P_g$.
Therefore, each $\psi_i$ extends to a univalent map to $Z$. Repeating arguments as in the proof of the alternative (A), Proposition \ref{p-bi}
(where one can take $V_0$ to be any Jordan domain such that $P_g\subset V_0\subset\overline{V_0}\subset U_0$ and $a\in\Gamma_0:=\partial V_0$) we obtain that $\lambda_i\to\infty$.
Equivalently, $\psi_i'(a_i)=\lambda_i^{-1}\to 0$. In turn, shrinking the radius $r$ of the disk $Z$ if necessary, this implies that $a_i\in\psi_{n_i}(Z)\subset\overline{\psi_{n_i}(Z)}\subset Z$ for all $i$ big enough.

\subsection{The case of connected $K$}\label{ss-k-conn}

Let $\pi: \D^*:=\{|w|>1\}\to\C\setminus K$
be the conformal isomorphism such that $\pi(\infty)=\infty$. Replace $g:U'\setminus K\to U\setminus K$ by
$\hat g:=\pi^{-1}\circ g\circ \pi: \hat U'\to\hat U$ where $\hat U'=\pi^{-1}(U'\setminus K)$,  $\hat U=\pi^{-1}(U\setminus K)$.
By Schwartz's reflection principle,
$\hat g$ extends to a holomorphic map $\hat g_*: \hat U'_*\to\hat U_*$
where $E_*$ denotes the union of a domain $E\subset \D^*$ containing $S^1$ in its boundary with its symmetric w.r.t. $S^1$ completed by $S^1$.

On the other hand, $\hat g_*^{-1}(S^1)=S^1$, hence $\hat g_*$ has no critical points. Therefore, all pullbacks of $\hat g_*^n$ are well-defined in any simply connected subdomain of $(\hat U_0)_*$. Fix such a subdomain $D$ that contains $\hat a$ and an arc $l$ of $S^1$.
As $\hat a_i\to \hat a$ and, for each $i$, $\hat a_i$ is an attracting fixed point of the map
$\hat\psi_i:\hat Z\to\C$ and $\hat\psi_i'(\hat a_i)\to 0$, we arrive at a sequence of holomorphic extensions $(\hat\psi_i)_*$
of $\hat\psi_i$ to $D$ such that
$(\hat\psi_i)_*\to \hat a$ in $D$. But this cannot hold for points of the arc $l$, by the complete invariance of $S^1$, a contradiction

The proof for the disconnected $K$ requires new ideas and is substantially more technical. It also holds in the case of connected $K$.
Let's indicate an idea returning to the connected $K$.
We don't consider the reflection of the map
$\hat g:\hat U'\to\hat U$ across $S^1$. Instead, let us fix a simply connected domain $\Omega$ in $\D^*=\{|z|>1\}$ such that, if $\hat Z=\pi^{-1}(Z)$, then $\hat a\in\hat Z\subset \Omega$ and $\partial \Omega\cap S^1$ contains a nontrivial arc $l$ of $S^1$ (in fact, the set $\partial \Omega\cap S^1$ having a positive length would be enough). Since all branches $\hat g_0^{-n}$ are well-defined
in $\Omega$, for every $i$, the local map $\hat\psi_i:\hat Z\to \hat Z$ extends to a univalent map $\hat\psi_i: \Omega\to\C$
where $\hat\psi_i(\hat a_i)=\hat a_i$ and $\hat\psi_i'(\hat a_i)\to 0$. The point is that this implies that $\hat\psi_i$ expands the length
of the arc $l$ by factors tending to $\infty$, see Appendix C. This is clearly a contradiction since all images of $l$ must lie in $S^1$.

\subsection{Proof of Theorem \ref{thm-m} for disconnected $K$}\label{ss-disc-pr}

By Theorem \ref{thm-entropy} and Corollary \ref{c-cap}, the logarithmic capacity of the completely invariant compact set $K$ is positive
so Corollary \ref{co-basic} applies.

Let us fix an open subset of $U_0$ as follows:

- it consists of a finitely many, but at least two, Jordan domains $Q_1,...,Q_m$, $m\ge 2$, such that $\overline Q_i$, $i=1,...,m$, are pairwise disjoint
and $K\subset \cup_{i=1}^m Q_i$. Moreover, the point $a\notin\cup_{i=1}^m \overline{Q_i}$.

Let $\eps>0$ be the distance between the set $\cup_{i=1}^m \partial Q_i$ and $K$.
For this $\eps$, we find a simply connected domain $W$ as in Corollary \ref{co-basic}, in particular, the connected set
$\pi(W)\subset \C\setminus K$ must lie
in one and only one of the domains $Q_1,...,Q_m$. W.l.o.g. one can assume that
$$\overline{\pi(W)}\subset Q_1.$$

Let us joint the point $a$ and a point $q\in\partial\pi(W)\setminus K$ by a simple arc $\alpha$ which, on the one hand,
lies in the range $U_0$ of $g$, and, on the other hand, except its end point $q$, lies outside
of $\overline{\pi(W)\cup K}$ as well as outside of all others $\overline{Q_i}$, $2\le i\le m$, and also crosses the Jordan curve
$\partial Q_1$ at a unique point. Choose a point $\hat q\in\partial W\cap\D$
such that $\pi(\hat q)=q$ and let $\hat\alpha$ be the lift of the curve $\alpha$ by the covering $\pi$ that begins at $\hat q$. Then
$\tilde\alpha$ is also a simple arc (in $\D$) with the end points $\hat q$ and some $\hat a\in\D$ such that $\pi(\hat a)=a$.
As $\pi$ is a local homeomorphism near $\hat a$, one can choose $r>0$ small enough so that there exists a topological disk $\hat Z$ around $\hat a$
such that $\pi: \hat Z\to Z$ where $Z=B(a,r)$ is a homeomorphism.

Note that $\hat\alpha\cap\overline{W}=\{\hat q\}$. Let us slightly "thicken" the curve $\hat\alpha$ to a narrow "tube" $T$ so that
$\Omega:=\hat Z\cup T\cup W$ is a {\it simply connected domain} in $\D$.
Then
$\pi(\Omega)=Z\cup \pi(T)\cup\pi(W)$ is a domain in $U_0\setminus K$
which is disjoint with $\overline{Q_i}$, for all $2\le i\le m$.
As $\alpha$ crosses $\partial Q_1$ at a single point, we can also assume that $Q:=\pi(\Omega)\cup Q_1=\pi(T)\cup Q_1$ is a simply connected domain,
which is disjoint with all others $Q_2,...,Q_m$.

For each $i$ big enough, we can "lift" the map $\psi_i: Z\to Z$ to a univalent map $\hat\psi_i: \hat Z\to\hat Z$
so that $\pi\circ \hat\psi_i=\psi_i\circ\pi$ on $\hat Z$.

{\bf Claim 0.} {\it $\hat\psi_i: \hat Z\to\hat Z$ extends to a well-defined
holomorphic function (also called $\hat\psi_i$) in the domain $\Omega$.}


Indeed, let $\hat\beta$ be any curve in $\Omega$ that starts at
$\hat a$. Then $\beta:=\pi\circ \hat\beta$ is a curve in $\pi(\Omega)$ that starts at $a$ and lies in $U_0\setminus K$.
As $g^{n_i}: U_{-n_i}\setminus K\to U_0\setminus K$ is an unbranched covering map,
the local branch $\psi_i: Z\to Z$ of $g^{-n_i}$ has an analytic continuation along the curve $\beta$, and it is lifted
by the unbranched covering map $\pi:\D\to\hat\C\setminus K$ to an analytic continuation of $\hat\psi_i$ along the curve $\hat\beta$.
Since the latter is any curve starting at $\hat a$ in the simply connected domain $\Omega$, by the Monodromy theorem,
the function $\hat\psi_i$ is well-defined holomorphic continuation of $\hat\psi_i:\hat Z\to\hat Z$
to the whole domain $\Omega$.

Moreover,

{\bf Claim 1.} {\it For each $i$, the function $\hat\psi_i:\Omega\to\C$ is univalent.}

Indeed, assume the contrary: $\hat\psi_i(x_1)=\hat\psi_i(x_2)$ for some $x_1\neq x_2$ in $\Omega$.
Let $\hat l$ be a simple curve in $\Omega$ joining $x_1$ and $x_2$.
As all branches of $g^{-n_i}$ are inverse branches of a well-defined map
$g^{n_i}:U_{-n_i}\to U_0$
and since
$Q=\pi(\Omega)\cup Q_1=\pi(T)\cup Q_1$ is a simply connected domain,
which is disjoint with some other $Q_j$, $j\neq 1$, the following must hold:
\begin{itemize}
\item $\pi(x_1)=\pi(x_2)$, i.e., $l:=\pi\circ \hat l$ is a closed curve in $\pi(\Omega)\subset Q\setminus K$,
\item
an analytic continuation of some local branch of $g^{-n_i}$ along the closed curve $l$ maps this curve again onto some closed
curve $L$,
\item
there exists a lift of $L$ by $\pi$ to a curve $\hat L$ which begins and ends at the same point $\hat y:=\hat\psi_n(x_1)=\hat\psi_n(x_2)$,
in other words the lift of $L$ is again a closed curve.
\end{itemize}
The latter condition implies that $L$ must be homotopic to a point in the domain $\hat\C\setminus K$.
On the other hand, as the curve $l$ lies in a simply connected domain $Q$ which is disjoint with some $Q_j$,
the curve $L$ must lie in one of the components of $g^{-n_i}(Q)$, which, in turn, are disjoint with all components
of $g^{-n_i}(Q_j)$. As each of the latter components contains points of $K$, the curve $L$ cannot enclose $\infty$ (in $\hat\C$), hence,
$L$ is homotopic to a point in the subdomain $\C\setminus K$ of $\C$.
Since $K$ is completely invariant, then $l$ is homotopic
to a point in $\C\setminus K$ too, hence, its lift $\hat l$ must be a closed curve, a contradiction. This proves Claim 1.

Let $F=\partial W\cap \partial \D$.
Thus, we have:
\begin{enumerate}
\item[($\hat 1$)] for every big $i$, $\hat\psi_i:\Omega\to\D$ is a univalent continuation of
$\hat\psi_i:\hat Z\to\hat\Z$,
\item[($\hat 2$)] $\hat\psi_i(\hat a_i)=\hat a_i$, $\hat a_i\to\hat a\in\Omega$ and $\hat\psi_i'(\hat a_i)=\psi_i'(a_i)\to 0$,
\item[($\hat 3$)] as $|F|>0$ and $K$ is completely invariant, for every big $i$ and almost every $w_0\in F$, $\hat\psi_i(w)$ has a an angular limit
$\hat\psi(w_0)$ as $w\to w_0$ in $\Omega$ and $\hat\psi_i(w_0)\in\partial\D$.
\end{enumerate}

We claim that ($\hat 1$)-($\hat 3$) are incompatible.


Indeed, let $R:\D\to\Omega$ be a Riemann map such that $R(0)=\hat a$. Let $f_i=\hat\psi_i\circ R$, a univalent map in $\D$, for each $i$.
As $\partial\Omega$ is a rectifiable Jordan curve, $R'\in H^1(\D)$.
Hence, $E:=R^{-1}(F)$ is a measurable subset of $S^1$ of a positive 1-dim measure $|E|>0$.

Let $f_i=\hat\psi_i\circ R$, a univalent map of $\D$ onto a bounded domain in $\D$, for each $i$.
Since $\hat\psi_i:\Omega\to\D$ are univalent and form a normal family, by ($\hat 2$),
$f_i(0)=\hat\psi_i(\hat a)\to \hat a$ where $\hat a\in\Omega\subset\D$, and
$f_i'(0)=\hat\psi_i'(\hat a)R'(0)\to 0$.
By ($\hat 3$), $f_i(E)\subset S^1$, for all $i$.
Therefore, one can choose $i_*$ such that
$$|E|\sqrt{\frac{(1-|f_{i_*}(0)|)^3}{(1+|f_{i_*}(0)|)|f_{i_*}'(0)|}}>2\pi,$$
in a contradiction with Corollary \ref{c-bdist} applied to $f=f_{i_*}$.

\section{Proof of Theorem \ref{thm-m}}\label{s-thm-m}
It now follows directly from Theorem \ref{thm-mainstep} along with Proposition \ref{p-bi}, where the alternative (B) must occur.

\section{Proof of Corollary \ref{c-uniq} and its Complement}\label{s-c-uniq}
Proof of (b). By Proposition \ref{p-bi}, for every Jordan curve $\Gamma$ in $U_0$ which encloses $P_g$
either (A) or (C) must hold. If (C) holds for a sequence of such $\Gamma_0$ that tends to $\partial U$, then, by (C), $g: U'\to U$ is trivial,
which is ruled out by a condition.
Therefore, for each Jordan curve $\Gamma$ which surrounds $P_g$ and is close enough to $\partial U_0$ the alternative (A) holds, that is, $\Gamma_0$ must intersect $R'$.
Now, consider a connected component $C$ of $R_g\cup\partial U_0$ which contains $\partial U_0$
It is enough to prove that it also contains points of $U_0$.
Assume the contrary, i.e., $C=\partial U$.
Fix some Jordan curve $\Gamma'$ as above (i.e., $\Gamma'$ surrounds $P_g$ and close to $\partial U_0$)
and let $C'$ be a connected component of $R_g\cup \Gamma'$ that contains $\Gamma'$.
Let $\Omega$ be the unbounded component of $\C\setminus C'$ and $D=(\Omega\cap U_0)\setminus R'$.
By the assumption, $D$ is an open connected set and $\partial U_0$ is a component of the boundary $\partial D$ while $\partial\Omega$ is another component of $\partial D$. Therefore, there exists a curve $\Gamma''$ in $D$ which surrounds $\partial\Omega$, hence, $P_g$ as well, and $\Gamma''$ is close to $\partial U_0$. Then $\Gamma''$ must intersect $R'$, a contradiction because $\Gamma''\subset D$ and $D$ is disjoint with $R'$.

Proof of (a). Item (1): by Theorem \ref{thm-m}, $K$ is the filled Julia set of a PL restriction $g:g_0^{-1}(W)\to W$.
By the Straightening theorem and the Fatou-Julia theory, $\partial K$ is equal to the limit set of repelling periodic points of $g: K\to K$. On the other hand, by Lemma \ref{l-compl-inv} applied to $K$, $R_n\cap K$ coincides with the set $\{z\in K: g^n(z)=z, |(g^n)'(z)|>1\}$. Hence, $R'\cap K=\partial K$, i.e., item (1). Item (2): assume, by a contradiction, that $x\in W\setminus K$ and $g^n(x)=x$ for some $n\in\N$ and $x\in U_{-n}\cap U_0$. On the one hand, $x\in g_0^{-n}(x)$. On the other hand, $\overline{g_0^{-1}(W)}\subset W$. Let $A_k=g_0^{-k}(W)\setminus \overline{g_0^{-k-1}(W)}$, $k=0,1,2,...$. Then $\{A_n\}_{n\ge 0}$ is a partition of $W\setminus K$. If now $x\in A_j$, for some $j\ge 0$, then
$x\in g_0^{-n}(x)\in A_{j+n}$ while $A_{j}\cap A_{j+n}=\emptyset$, a contradiction.
Item (3) follows directly from Theorem \ref{thm-mainstep}.

{\it Proof of the Complement.} After the proven Corollary and Remark \ref{r-degree}, we only need to show the claim about iterates of $x$.
Let $x\in V'\setminus K$ be such that $O(x)$ is well-defined but $\omega(x)\nsubseteq \partial V$.
In other words, there exist a sequence $n_i\to\infty$ a compact $E\subset V$ so that $g^{n_i}(x)\in E$, for all $i$.
Choose a Jordan domain $V_0$ such that $P_g\cup E\subset V_0\subset\overline{V_0}\subset V$.
As $P_g\subset K$, by Theorem \ref{thm-mainstep} along with Proposition \ref{p-bi}, the alternative (B) of Proposition \ref{p-bi} must hold, i.e.,
for some $n\in \N$, $V_{-n}:=g^{-n}(V_0)$ is contained in $V_0$ with its closure. By (B), then $K=K_G:=\cap_{k\ge 0}g^{-kn}(V)$.
On the other hand, write $n_i=l_i n + r_i$, with integers $l_i, r_i$ such that $l_i\to\infty$ and $0\le r_i<n$. Passing to a subsequence, one can assume that $r_i=r$ for all $i$. Then $g^{l_i n}(g^r(x))\in E\subset V_0$, for all $i$, hence,
$g^r(x)\in\cap_{i\ge 0}g^{-l_i n}(V_0)\subset\cap_{k\ge 0}g^{-kn}(V_0)$, where the latter equality follows from the fact that the sequence of domains
$\{g^{-kn}(V_0)\}_{k\ge 0}$ is decreasing. Thus $g^r(x)\in K$ so that $x\in g^{-r}(K)=K$, a contradiction.

\section{Proof of Theorem \ref{thm-gen-bi}}\label{s-thm-gen-bi}
The proof follows almost literally the one of Theorem \ref{thm-m} for disconnected $K$,
where very minor changed are required and only in Proposition \ref{p-bi} and its proof.
Curves $\Gamma_n$ turn into multi-curves. In the alternative (C), $\overline{V_0}$ should be a subset of a component of $V_{-n}$.
In the part of the proof of Proposition \ref{p-bi} where (A) holds, the normality argument in the case $P_g=\emptyset$ can be easily amended;
in the part where (B) holds, $V$ should be a neighborhood of $K$ consisting of a finitely many simply connected
components and the {\it Second proof} of the construction of $V$ such that $g^{-1}(V)$ is compactly in $V$ should be chosen.

\section{BI maps associated to the renormalization}\label{s-examples}
For a renormalizable quadratic polynomial $f(z)=z^2+c$,
let $\mathcal{SR}$ be the consecutive sequence of periods of {\it simple} renormalizations of $f$, see \cite{mcm} for this and related definitions.
Fix some $n\in\mathcal{SR}\setminus \{1\}$. Let $J_n$ be the "little" Julia set (i.e., the filled Julia set of the corresponding PL map around $0$) and $P_n:=P_f\cap J_n$ (where $P=P_f$) be the corresponding "little" postcritical set. Let $\gamma_n$ be the closed geodesic of the hyperbolic Riemann surface $M=\hat{\C}\setminus P_f$
which is isotopic in $M$ to a simple closed curve $\Gamma_n$ separating $J_n$ from $P_f\setminus J_n$, and let $L_n=l(\gamma_n)$ be the hyperbolic length
(denoted by $l$) of $\gamma_n$.
Let us construct inductively a sequence of simple closed curves $\gamma_n^i$ and the corresponding sequence of simply connected domains $V_n^i$ with $\gamma_n^i=\partial V_n^i$, $i=0,1,...$, such that the following hold: $\gamma_n^i$ is isotopic to $\Gamma_n$, separates $P_n$ from all others $f^r(P_n)$, $0<r<n$,
and $f^n: V_n^{i+1}\to V_n^i$ is a proper map of degree $2$ with a single critical point $0$. (One more property appears below.)
.
We begin with $\gamma_n^0=\gamma_n$ and $V_n^0$ to be the Jordan domain bounded by $\gamma_n$. Assuming $\gamma_n^i$, $V_n^i$ have been constructed, we
pull back $V_n^i$ along the orbit $f^n(0)\in V_n^i$, $f^{n-1}(0),..., f(0)$.
For $j=1,...,n-1$, the pullback from $f^n(0)$ to $f^{n-j}(0)$ is disjoint with $P_n$
(in particular, with $0$) because otherwise $V_n^i$ would contain points of $P_n$ and $f^j(P_n)$.
Hence, they are univalent pullbacks. Also, this pullback contains $f^{n-j}(P_n)$
because $\gamma_n^i$ is isotopic to $\Gamma_n$. Then one more pullback (from $f(0)$ to $0$) gives us a domain $V_n^{i+1}$ which contains $P_n$ and is disjoint with all others $f^j(P_n)$, $0<j<n$.
We let $\gamma_n^{i+1}:=\partial V_n^{i+1}$, which is then isotopic to $\Gamma_n$.
The hyperbolic length of $\gamma_n^i$ is then bigger than or equal to $L_n/2^i$ because the map $f^n:\C\setminus P\to \C\setminus f^{-n}(P)$ is a local isometry (see \cite{mcm} for details).

By the Color Theorem (e.g. \cite{mcm}), there is a decreasing function $L\mapsto m(L)\in (0,+\infty)$ and an embedded annulus $C_n\subset M$ of the modulus $\mod(C_n)\ge m(L)$ such that $\gamma_n$ is its core curve. Pulling back $C_n^0:=C_n$ as we did for $V_n^0$, we get the following, for every $i=1,2,...$: (1) an embedded in $M$ annulus
$C_n^i$ with $\mod(C_n^i)\ge m(L_n)/2^i$ such that $\gamma_n^i$ is the core curve of $C_n^i$, and (2) a Jordan domain $\tilde V_n^i=V_n^i\setminus C_n^i$ which still contains $P_n$.

It follows from (1) that $\gamma_n^i$ is a $K=K(L_n/2^i)$-quasicircle, where $K: m\mapsto K(m)\in(1,\infty)$ is a decreasing function.

The infinitely renormalizable $f$ (i.e, assuming $\#\mathcal{SR}=\infty$) is called {\it robust} \cite{mcm} if
$\liminf_{n\in\mathcal{SR}} L_n<\infty$.
The case $\liminf=0$ is reduced to the case when $f$ has so called unbranched complex bounds \cite{mcm} (this means that there exist, along an infinite subsquence of $\mathcal{SR}$, polynomial-like maps with filled Julia sets $J_n$ such that their fundamental annuli are in $M_n$ and moduli are uniformly away from zero).

The other, more interesting for us case (which we assume to hold from now on to the end of this Section) is when there exist $0<L_{-}<L_+<\infty$ such that
$$L_-<L_n<L_+ \mbox{ along an infinite sequence } \mathcal{R}=\{n_i\}\subset\mathcal{SR}.$$
Such maps exist, see \cite{CP}.

Then we construct a limit BI map as follows.
We start by observing that:
\begin{enumerate}
\item $\gamma_n^i$ is a $K_i$-quasicircle, for some $K_i>1$ and all $n\in \mathcal{R}$,
\item $L_{-}/2^i\le l(\gamma_n^i)\le L_+/2^i$, for all $n\in \mathcal{R}$ and all $i$,
\item for $n\in \mathcal{R}$ and $i=0,1,...$, let $B_n^i$ be the maximal ball centered at $0$ and is contained in $V_n^i\cap V_n^{i+1}$.
Using the previous item
and repeating the proof of \cite{mcm}, Theorem 10.12, we obtain that, for each $i$ there exists some $D_i>0$ such that, for all $n$:
$$D_i \diam(V_n^i)\le \diam(B_n^i)\le \diam(V_n^i),$$
$$D_i \diam(V_n^{i+1})\le \diam(B_n^i)\le \diam(V_n^{i+1}),$$
where $\diam(X)$ is the Euclidean diameter of $X\subset\C$,
\item
inequalities of Item (3) imply that, for every $i$, there exists $D_i'>0$, such that, for all $n$:
$$\frac{1}{D_i'}\le \frac{\diam(B_n^i)}{\diam(B_n^0)}\le D_i'$$
\end{enumerate}
Let
\begin{equation}\label{ag}
A_n(z)=\frac{z}{\diam(B_n^0)}, \ \ G_n=A_n\circ f^n\circ A_n^{-1}.
\end{equation}
We note that, by Item (1), for every $i$ and for any Caratheodory limit pointed domain $(W^i, u^i)$ of the sequence of pointed domains $\{(A_n(V_n^i), u_n^i)\}_{n=1}^\infty$,
the curve $\partial W^i$ is a quasicircle and is the Hausdorff limit of the curves $\partial A_n(V_n^i)$.

All this allows us to repeat the proof of \cite{mcm}, Theorem 10.13 and obtain a subsequence $\mathcal{R}_0\subset \mathcal{R}$, such that, along this subsequence:
\begin{enumerate}
\item[(i)] $(A_n(V_n^1), 0) \to (U_{-1}, 0)$ and $(A_n(V_n^0), G_n(0)) \to (U_{0}, v_0)$, where the convergence is in the sense of Caratheodory topology,
\item[(ii)]
$G_n \to g_0$ uniformly on compacts in $U_{-1}$,
where $g_0: U_{-1}\to U_0$ is a holomorphic proper map of degree two, $v_0=g_0(0)$ and $g_0'(0)=0$,
\item[(iv)]
the sequence of compacts $P_n^*:=A_n(P_n)$, $n\in R_0$, converges to a compact $X$
in the {\it Hausdorff distance}.
\end{enumerate}
We claim that, for every $k=1,2,...$,
\begin{equation}\label{gn0g0}
G_n^k(0)\to g_0^k(0)
\end{equation}
as $n\to\infty$ along $n\in \mathcal{R}_0$.
Indeed, as $G_n\to g_0$ uniformly on compacts in $U_{-1}$, it follows from Cauchy's formula that if $z_n\to z$,
where $z_n\in A_n(V_n^1)$ and $z\in U_{-1}$, then $G_n(z_n)\to g_0(z)$. As $G_n(0)\to g_0(0)$, this implies inductively (\ref{gn0g0}).
By (\ref{gn0g0}), $g_0^k(0)\in X$ for all $k$, hence, also $P_g\subset X$. On the other hand, by the considerations about annuli $C_n^i$, for each $n$ and each $i$,
there is an annulus of a modulus $\ge L_{-}/2^{i+1}$ between $\gamma_n^i$ and $P_n$.
This (for $i=0,1$) implies that $X$ is a compact subset of $U_{-1}\cap U_0$.

Now we replace in the above considerations the pair $(V_n^0, V_n^1)$ by $(V_n^1, V_n^2)$ (keeping the maps $A_n$, $G_n$ the same as before
although restricted to the corresponding domains) such that (i)-(iv) hold. We find a subsequence $\mathcal{R}_1\subset \mathcal{R}_0$,
such that, as $n\to\infty$ in $\mathcal{R}_1$, (i)-(iv) hold where $(U_0, U_{-1})$ is replaced by $(U_{-1}, U_{-2})$, and, as before,
$X\subset U_{-2}\cap U_{-1}$. We also have a new proper map of degree two $g_1: U_{-2}\to U_{-1}$ such that $G_n\to g_1$ uniformly on compacts
in $U_{-2}$.

Then we do the same replacing $(V_n^1, V_n^2)$ by $(V_n^2, V_n^3)$ and, passing to a subsequence $\mathcal{R}_2\subset \mathcal{R}_1$, we getting the corresponding limit map $g_2: U_{-3}\to U_{-2}$, $X\subset U_{-3}\cap U_{-2}$, etc., obtaining a sequence of simply connected domains $U_{-i}$ and proper maps $g_i: U_{-i-1}\to U_{-i}$ of degree two, for $i=0,1,...$, such that $X\subset\cap_i U_{-i}$ and
$G_n\to g_i$ uniformly on compacts along the sequence $\mathcal{R}_i$.
Observe that, if $0\le i<j<\infty$ and $x\in U_{-i-1}\cap U_{-j-1}$,
then, along
the sequence $\mathcal{R}_j$, $G_n(x)\to g_i(x)$ and $G_n(x)\to g_j(x)$. Thus $g_i=g_j$ on $U_{-i-1}\cap U_{-j-1}$, so that the map $g:\cup_{i\ge 1} U_{-i}\to \cup_{i\ge 0} U_{-i}$
where $g|_{U_{-i-1}}=g_i$, is well-defined.

Finally, by Cantor's diagonal procedure we arrive at an infinite sequence $\mathcal{R}_*$ such that, as $n\to\infty$ in $\mathcal{R}_*$,
for every $i=0,1,...$, the following claims (i*)-(v*) hold:
\begin{enumerate}
\item[(i*)] $(A_n(V_n^{i+1}), 0) \to (U_{-i-1}, 0)$ and $(A_n(V_n^i), G_n(0)) \to (U_{-i}, v)$, where the convergence is in the sense of Caratheodory topology, as well as in the Hausdorff distance (while passing to the closures).
\item[(ii*)]
$G_n \to g$ uniformly on compacts in $U_{-i-1}$,
where $g: U_{-i-1}\to U_{-i}$ is a holomorphic proper map of degree two, $v=g(0)$ and $g'(0)=0$,
i.e., $C_g=\{0\}$,
\item[(iii*)]
the sequence of compacts $P_n^*:=A_n(P_n)$, $n\in \mathcal{R}_*$, converges to a compact $X$
in the {\it Hausdorff distance} where $X\subset \cap_{i\ge 0} U_{-i}$.
As $0\in P_n$, for all $n$, hence, $0=A_n(0)\in P_n^*$ for all $n$, hence, $0\in X$.
\item[(iv*)] $G_n^k(0)\to g^k(0)$
for every $k=1,2,...$,
Therefore, $P_g:=\overline{\{g^k(0)\}_{k=1}^\infty}\subset X\subset \cap_{i\ge 0} U_{-i}$.
\end{enumerate}
Thus, $g: \cup_{i\ge 1} U_{-i}\to \cup_{i\ge 0} U_{-i}$ is the required limit BI map.

\appendix

\section{Entropy of locally analytic maps}\label{a-entropy}
The following result is contained essentially in \cite{MP}, \cite{P}:
\begin{theorem}\label{thm-entropy}

(1) Let $X\subset\C$ be a compact subset of the plane
and $f:X\to X$ be a continuous map which extends to a holomorphic map $f:W'\to W$ of some neighborhoods $W', W$ of $X$, such that
$f^{-1}(X)=X$ and $\deg f|_X=
N\ge 2$.
Assume that
$X$ is not equal to the union of cycles of critical points of $f$.
Then $h_{top}(f)\ge \log N$, for the topological entropy $h_{top}(f)$ of $f:X\to X$.

(2) If, additionally, $f: W'\to W$ is a proper map of degree $N$ where $W$ is simply connected, then the conclusion holds unless $W'$ is (simply) connected and $X$ is a singleton $\{c\}$ where $c$ is a fixed critical point of $f$ of multiplicity $N$.
\end{theorem}

\begin{proof}
(1) We follow closely \cite{P} (where $X$ is the Julia set of a rational function $f$) and \cite{MP} (where $f$ is a smooth self-map of a compact smooth manifold) taking into account that our map $f$ is defined only in a small neighborhood of $X$.
Shrinking $W$ if necessary one can assume that all critical points of $f:W'\to W$ lie in $X$.
Let $C$, $C_s$ be the sets of non periodic and periodic critical points of $f:X\to X$, respectively. (It is possible though that either $C_s$ or both sets $C, C_s$ are empty.)
$C(\sigma)=\cup_{c\in C}B(c,\sigma)$ be the $\sigma$-neighborhood of $C$ and $C_s(\sigma)$ be a neighborhood of $C_s$
which is forward invariant
($f(C_s(\sigma))\subset C_s(\sigma)$) and is
contained in the $\sigma$-neighborhood of $C_s$
(all distances are Euclidean). Note that $C(\sigma)$ is disjoint with $C_s(\sigma)$ for all $\sigma>0$ small enough.
We prove:

($\star$) {\it given $t\in (0,1)$, there exist $\sigma>0$ small enough and $k_0$ such that for every $k\ge k_0$ and every $y\in X$,}
if $y, f(y), f^2(y),...,f^k(y)\notin C_s(\sigma)$ then
\begin{equation}\label{dens}
\#\{j: 0\le j < k, f^j(y)\in C(\sigma)\}<t k.
\end{equation}

First, if $y$ is attracted by an attracting (not superattracting) or parabolic cycle, such that some $c\in C$ is also in the basin of attraction of the same cycle (the number of such cycles is at most $\# C$), then (\ref{dens}) holds because the left-hand side of (\ref{dens}) is bounded,
uniformly over $y$ provided $\sigma$ is small enough.
Hence, one can assume that

($\star$$\star$) {\it $y$ is not attracted by a cycle that attracts also a point of $C$}.

In this case, since $C, C_s$ are finite sets, (\ref{dens}) would follow if we show that
\begin{equation}\label{pomozh}
\inf\{m: z, f^m(z)\in B(c,\sigma) \mbox{ for some } c\in C, z\in X\}\to\infty
\end{equation}
as $\sigma\to 0$.
To this end, let $r>0$ be such that $X_r:=\cup_{a\in X}B(a,r)\subset W'$ and $M=\sup\{|f'(z)|: z\in X_r\}<\infty$.
Given $c\in C$, there exist $\rho>0$, $L>0$, $l\ge 2$ (the order of $c$) such that $L^{-1} |z-c|^{l-1}<|f'(z)|< L|z-c|^{l-1}$ for all
$z\in B(c, \rho)$. We prove that there exists $L_*$ that depends on $M$, $L$ and $l$ such that $m>L_*/\log(1/\sigma)$ whenever $\sigma<\min\{r/8, \rho\}$ and
$z, f^m(x)\in B(c,\sigma)$ for some $x\in X$.
It is enough to prove that
$$L (5\sigma)^l M^m \ge 2\sigma.$$
If we assume the contrary, then, on the one hand, all images $f^j(B(x, 4\sigma))$, $j=0,1,2,...,m$, are well-defined, i.e.
are contained in $X_r\subset W'$. On the other hand, as $|x-f^m(x)|<2\sigma$, we obtain that
$\overline{f^m(B(x, 4\sigma))}\subset B(x,4\sigma)$ which implies that $f^{mk}$ tends as $k\to\infty$ to an attracting fixed point of $f^m$
uniformly on compacts in $B(x,4\sigma)$. Note that $c\in B(x, 4\sigma)$, therefore, $x$ and $c$ are attracted by an attracting cycle of $f$,
in contradiction with ($\star$$\star$).
This
proves ($\star$).

We proceed by fixing any $\alpha\in (0,1)$. For $t=1-\alpha$, find $\sigma>0$ small enough and $k_0$ such that ($\star$) holds
and such that the set $B(\sigma):=X\setminus C(\sigma)$ is not empty (the latter is possible by the condition on $X$). Then, for every $y\in X$ and every $k>k_0$,
\begin{equation}\label{dens1}
\#\{j: 0\le j < k, f^j(y)\in B(\sigma)\} > \alpha k
\end{equation}
whenever $y,f(y),...,f^k(y)\notin C_s(\sigma)$.
On the other hand, every $x\in B(\sigma)$ has a small neighborhood $U_x$ such that $f:U_x\to\C$ is a homeomorphism.
Let $\delta>0$ be the Lebesgue number of the covering $(U_x)$ of $B(\sigma)$. Then
\begin{equation}\label{sep}
x_1,x_2\in B(\sigma), \
|x_1-x_2|<\delta \Rightarrow f(x_1)\neq f(x_2).
\end{equation}
Now, take any $x\in B(\sigma)\setminus C_s(\sigma)$.
As $C_s(\sigma)$ is forward invariant, then $f^{-k}(x)\cap C_s(\sigma)=\emptyset$ for all $k\ge 0$.
It follows from the condition on $X$ that we can choose $x$ not to be in the forward orbit of any critical point of $f$.

We are in a position to repeat literally the end of the proof of \cite{MP} (where $B=B(\sigma)$).

A set $P\subset X$ is called $\delta$-separated if $|z_1-z_2|>\delta$ for all distinct $z_1,z_2\in P$
and $P$ is $(k,\delta)$ separated if
$\max\{|f^j(z_1)-f^j(z_2)|>\delta, j=0,...,k-1\}$ for all distinct points $z_1,z_2\in P$
(in particular, any singleton is $(k,\delta)$-separated).
Consider the backward orbit $BO(x)=\cup_{k=0}^\infty f^{-k}(x)$ of $x$.
For every $y\in BO(x)$ there exists a $\delta$-separated set $P(y)\subset f^{-1}(y)$ such that either $P(y)=\{z\}$ for some $z\notin B(\sigma)$ or
$P(y)\subset B(\sigma)$ and $\#P(y)=N$. Fix $k>k_0$ and define by induction:
$$Q_0=\{x\}; \ \ Q_{j+1}=\cup_{y\in Q_j} P(y), \ \ j=0,1,...,k-1.$$
Then $Q_k$ is $(k,\delta)$-separated (induction in $j=0,...,k$). Let us show that $\# Q_k\ge N^{\alpha k}$.
For $y\in Q_j$, let $Gen(y,j)=\#\{i: 0\le i<j, f^i(y)\in B(\sigma)\}$. Let $m=[\alpha k]+1$ (where $[T]$ is the entire part of $T>0$).
Then, by (\ref{dens1}), there are $N^m$ pairs $(y,j)$ such that $y\in B(\sigma)\cap Q_j$ and $Gen(y,j)=m$.
For each such a pair there exists $z\in Q_k$ for which $f^{k-j}(z)=y$ and such points are distinct for distinct pairs.
Therefore, $\# Q_k\ge N^m\ge N^{\alpha k}$. We have shown that for every $\alpha\in (0,1)$ there exist $k_0$ and $\delta>0$ such that the maximal cardinality $r_k(f,\delta)$
of a $(k,\delta)$-separated subset of $X$ is bigger than or equal to $N^{\alpha k}$, for all $k>k_0$. Thus
$h_{top}(f)=\lim_{\delta\to 0}\limsup_{k\to\infty}\frac{1}{k}\log r_k(f,\delta)\ge \alpha \log N$, for all $\alpha\in (0,1)$. We are done.

(2) Suppose that $W$ is simply connected and $X\subset W'\cap W$ consists of cycles of critical points of $f$. Then $\#f^{-1}(x)=1$
for all $x\in X$ which, along with the Riemann-Hurwitz formula for the map $f:U\to W$ where $U$ is a component of $W'$ yield that $W'$ consists of a single component and $X$ is just a point
(which is then a fixed critical point of $f: W'\to W$ of the order $d$).

\end{proof}
\begin{remark}\label{r-nodegree}
The conclusion (with the same proof) holds if the condition $\deg f|_X=N$ is replaced by a weaker one: $\#f^{-1}(x)\ge N$ for each $x\in X$ (counting multiplicity).
\end{remark}

\begin{coro}\label{c-cap}
Under the conditions of Theorem \ref{thm-entropy},
the logarithmic capacity of the compact set $X$ is positive.
\end{coro}
\begin{proof} By the Variational Principle, $\log N\le h_{top}(f)=\sup h_\mu(f)$, where supremum is taken over the set $M(f)$ of all
$f$-invariant Borel probability measures on $X$ and $h_\mu(f)$ is the entropy of $f$ w.r.t. $\mu$, see e.g. \cite{PU}. In particular, there exists an ergodic $\mu\in M(f)$ with $h_\mu(f)>0$. This implies that the Lyapunov exponent $\chi_\mu(f):=\int \log|f'| d\mu>0$.
Therefore, $HD(\mu)=h_\mu(f)/\chi_\mu(f)>0$, where $HD(\mu)$ is the infimum of Hausdorff dimensions over all Borel $Y\subset X$ such that
$\mu(Y)=1$. Hence, Hausdorff dimension of $X$ is positive. In turn, this implies (see e.g. \cite{tsuji}) that the logarithmic capacity of $X$ is positive as well.
\end{proof}
\begin{remark}\label{r-manyper}
By \cite{PU}, Theorem 11.6.1, any such ergodic $\mu\in M(f)$ is approximated by measures supported on Cantor expanding repellers of $f$.
In particular, there exist a lot of periodic orbits in any neighborhood of $X$. However, in general they are off $X$.
\end{remark}

\section{Uniformization}\label{a-unif}
\begin{theorem}\label{thm-green} (see e.g. \cite{tsuji})
Let $X$ be a full compact set in $\C$ that contains at least $3$ points.
Let $\pi:\D=\{|w|<1\}\to\hat C\setminus X$ be a universal holomorphic covering map and $\Gamma=\{\gamma\}$ be the corresponding Fuchsian group
of Mobius transformations of $\D$ without elliptic elements such that $\C\setminus X\simeq \D/\Gamma$. Normalize $\pi$ such that $\pi(0)=\infty$.
The following conditions are equivalent:
\begin{enumerate}
\item logarithmic capacity $c(X)$ of $X$ is positive,
\item $\Gamma$ is of convergence type, that is
$$\sum_{\gamma\in\Gamma}(1-|\gamma(w)|)<\infty$$
for some (hence, any) $w\in\D$. In this case, the infinite Blaschke product
$$B(w)=\Pi_{\gamma\in\Gamma}\frac{|\gamma(0)|}{\gamma(0)}\frac{w-\gamma(0)}{1-\overline{\gamma(0)}w}$$
is a holomorphic function in $\D$ such that $|B\circ \gamma(w)|=|B(w)|$ for all $w\in\D$, $\gamma\in\Gamma$,
$|B(w)|<1$ for all $w\in\D$ and, for almost every $w_0\in\partial\D$, the angular limit $\lim |B(w)|=1$ as $w\to w_0$ exists.

Moreover,
projecting the function $-\log|B|$ to $\C\setminus X$ we obtain
Green's function $G_X$ of $\C\setminus X$,
that is,
$$G_X(z)=-\log|B(w)|,$$
for some (hence, every) $w\in\D\setminus \{\gamma(0): \gamma\in\Gamma\}$, such that $z=\pi(w)$. Then $G_X$ is well-defined everywhere positive harmonic function in $\C\setminus X$,
$G(z)-\log |z|\to -\log(c(X))$ as $z\to\infty$ and, for almost every $w_0\in\partial\D$, the following angular limit exists:
\begin{equation}\label{eq-ae}
\lim_{w\to w_0}G_X(\pi(w))=0.
\end{equation}
\end{enumerate}
\end{theorem}

\begin{definition}\label{d-triangle}
For an angle $\beta\in (0, \pi/2)$,
$r\in (\sin\beta, 1)$,
and a point $a\in\partial\D$, let $\Delta_a(\beta, r)$ be
a connected component of the set $\{z: r<|z|<1, |\arg(z-a)|<\beta\}$ which contains $a$ in its boundary.
\end{definition}
In other words, $\Delta_a(\beta, r)$ is
an open "triangle" at a vertex $a$ bounded by an arc of the circle $|z|<r$ and segments of two rays through $a$ which form the angle $\beta$
with the radius $[0, a]$ of the unit circle.

Recall that $|F|$ denotes the Lebesgue 1-dim measure (length) of a measurable subset $F\subset\partial\D$

We have:
\begin{coro}\label{co-basic}
Let $X$ be a compact with positive logarithmic capacity.
Then for every $\eps>0$ there exists a domain $W\subset\D$ as follows:
\begin{enumerate}
\item [(1)]$W$ is simply connected,
\item [(2)] $\partial W$ is a rectifiable Jordan curve,
\item [(3)] $|\partial W\cap\partial \D|>0$,
\item [(4)] $\pi(W)$ lies in the $\eps$-neighborhood of $X$,
\item [(5)] $\pi(w)\to X$ uniformly as $w\to \partial\D$ inside of $W$.
\end{enumerate}
\end{coro}
\begin{proof} Let $X_\eps=\{z\in\C\setminus X: \dist(z, X)<\eps\}$. Define
$$l_\eps=\inf\{G_X(z): z\in\C\setminus X_\eps\}.$$
Then $l_\eps>0$ and if $G_X(z)<l_\eps$ for some $z\in\C\setminus X$, then $z\in X_\eps$.

Let $b=\exp(-l_\eps)$. Then $|B(w)|>b$ implies $\pi(w)\in X_\eps$.

Now we repeat partly a Lusin-Privalov construction (as in the proof of Privalov's theorem, e.g. \cite{tsuji}).
Fix an angle $\beta\in (0,\pi/2)$ and given $a\in \partial\D$ and $n\in\N$, such that
$1-1/n>\sin\beta$,
define
$$H_n(a)=\inf\{|B(w)|: w\in \Delta_a(\beta,1-1/n)\}.$$
By Theorem \ref{thm-green},
$$\lim_{n\to\infty}H_n(a)=1$$
for a.e. $a\in\partial\D$.
Hence, by Egorov's theorem, there exists a closed
subset $F\subset \partial\D$ of positive measure such that $H_n(a)\to 1$
as $n\to\infty$ uniformly in $a\in F$. In particular, for some $n_0\in\N$,
$$|B(w)|>b$$
whenever
$$w\in W_*=\cup_{a\in F}\Delta_a(\beta, 1-1/n_0).$$
One can assume that $F$ lies in an arc of $\partial \D$ of small enough length and $n_0$ is big enough
so that the part of the boundary
of $W_*$ which lies on the circle $|w|=1-1/n_0$ is a proper part of this circle.
Note that $\partial W_*\cap\partial\D=F$. Hence, there exists a connected component $W$ of the (open) set $W_*$
(note that $W=W_*$ if $W_*$ is connected itself)
and a closed subset $F_0$ of $F$ of positive measure as follows:
$$W=\cup_{a\in F_0}\Delta_a(\beta, 1-1/n_0),$$
$|B(w)|>b$ for every $w\in W$ and $|B(w)|\to 1$ as $w\to F_0$ uniformly in $W$.
Then properties (3)-(5) hold for $W$.
Let's show that (1)-(2) hold too. By the construction, $\partial W\cap \{|w|=1-1/n_0\}$ is a proper arc
$\{(1-1/n_0)\exp{it}: t_-\le t\le t_+\}$ of the circle $|w|=1-1/n_0$. Then the closure of $W$ is a subset of
$R:=\{w=r\exp(it): 1-1/n_0\le r\le 1, t_-\le t\le t_+\}$. Therefore, there exists a neighborhood $R^+$ of $R$ and a diffeomorphic in $R^+$ branch ${\bf Log}$ of the
map $w\mapsto i\overline{\log(w)}$ such that  (slightly abusing notation and replacing $W$ by ${\bf Log}(W)$) one can assume that
$F_0$ is a compact subset of $\R$ and
$$W=\{z=x+iy: 0<y<h, |\arg(z-x_0)|<\pi/2-\beta \mbox{ for some } x_0\in F_0\},$$
where $h=-\log(1-1/n_0)$.
Then it is easy to see that the boundary of $W$ is connected, which implies that $W$ is simply connected.
Now, if $I=[x_1,x_2]\subset \R$ is the smallest closed interval containing $F_0$
and $I\setminus F_0=\cup_{j=1}^\infty I_j$ where $I_j$ are pairwise disjoint open intervals, then the boundary
of $W$ consists of:
the "upper" boundary interval $\{x+ih: x_1-h\tan(\beta)\le x\le x_2+h\tan(\beta)\}$,
"left" and "right" intervals $[x_1-h\tan(\beta)+ih, x_1]$ and $[x_2, x_2+h\tan(\beta)+ih]$,
the compact set $F_0\subset \R$ and a "saw"-like curve consisting of a sequence of two sides of equilateral triangles with bases $I_j$ whose total
length is bounded by $\sum_j |I_j|/\sin(\beta)=(|I|-|F_0|)/\sin(\beta)$. Thus $\partial W$ is a rectifiable curve.
\end{proof}

\section{Boundary distortion of sets under univalent self-maps of a disk}\label{a-pfl}

Let $f:\D\to\D$ be a univalent map of the unit disk into itself. The following are known theorems, see \cite{pomm}.
\begin{itemize}
\item For a.e. $\zeta\in S^1=\partial\D$, $f(z)$ has an angular limit $f(\zeta)$.
\item $f$ has a finite angular derivative at $\zeta\in S^1$ if and only if
the function $f'(z)$ has an angular limit $f'(\zeta)$ at $\zeta$.
\item
If, for some $\zeta\in S^1$, $f(\zeta)\in S^1$, then (the angular derivative) $f'(\zeta)$ exists and $1\le \frac{\zeta}{f(\zeta)}f'(\zeta)\le \infty$
(the Wolff-Denjoy Theorem).
\item
If $E\subset S^1$ is measurable and $f(E)\subset S^1$, then $f'(\zeta)$ is finite for a.e. $\zeta\in E$ (a corollary from
the McMillan Twist Theorem).
\end{itemize}
In the next lemma, the inequality (\ref{eq-b2}) of part 1 is proved in \cite{nehari}\footnote{Although not explicitly stated in \cite{nehari}, the inequality (\ref{eq-b2}) is attained as the limit case, as $z\to 1$, $z\in\R$, in the inequality at the top of
p.259 assuming w.l.o.g. that $\zeta=1$.}, \cite{sol}, \cite{pomm-vas}.
Part 2 is an easy consequence, see below.
\begin{lemma}\label{l-bounds}
Let $f:\D\to\D$ be univalent. For some $\zeta\in S^1$, suppose that the angular limit $f(\zeta)$ exists and
$f(\zeta)\in S^1$.  Assume that the angular derivative $f'(\zeta)$ is finite.
\begin{enumerate}
\item[(b1)] If $f(0)=0$, then
\begin{equation}\label{eq-b2}
|f'(0)|^{-1/2}\le |f'(\zeta)|.
\end{equation}
\item[(b2)] in general,
\begin{equation}\label{eq-b3}
\frac{(1-|f(0)|)^{3/2}}{(1+|f(0)|)^{1/2}}|f'(0)|^{-1/2}\le |f'(\zeta)|.
\end{equation}
\end{enumerate}
\end{lemma}
Proof of (b2):
let $g(z)=\frac{f(z)-A}{1-\bar A f(z)}$ where $A=f(0)$.
As
$g'(z)=\frac{f'(z)(1-|A|^2)}{(1-\bar A f(z))^2}$, then $g'(0)=\frac{f'(0)}{1-|A|^2}$ and
$|g'(\zeta)|=\frac{|f'(\zeta)|(1-|A|^2)}{|1-\bar A f(\zeta)|^2}\le \frac{|f'(\zeta)|(1-|A|^2)}{(1-|A|)^2}=\frac{|f'(\zeta)|(1+|A|)}{(1-|A|)}$.
By (\ref{eq-b2}), this implies (\ref{eq-b3}).

The following lemma and its corollary state that, if $E, f(E)\subset S^1$, then the size of $E$ is small provided $|f'(0)|$ is small
and $f(0)$ is away from the unit circle. Exact bounds are obtained in \cite{sol}, under the assumption that $E$ is contained in an open set $E_0\subset S^1$ such that $f(E_0)\subset S^1$ (and $f(0)=0$).
\begin{lemma}\label{l-bdist}
Let $f:\D\to\D$ be univalent. Let $E\subset S^1$ be a measurable set such that $f(E)\subset S^1$, and $|E|>0$.
Then for every $\epsilon>0$ there exists a closed set $L\subset E$ such that $|L|>|E|-\epsilon$, $f(L)$ is measurable and
\begin{equation}\label{eq-bdist}
|f(L)|\ge |L| \sqrt{\frac{(1-|f(0)|)^3}{(1+|f(0)|)|f'(0)|}}.
\end{equation}
\end{lemma}
\begin{proof}
As $f'(\zeta)$ exists and is finite for a.e. $\zeta\in E$, and $f'(\zeta)=\lim_{n\to\infty}f'((1-1/n)\zeta)$
where $f'$ is continuous in $\D$, then
$f':E\to\C$ is measurable. By Lusin's theorem, there exists a closed subset $E'$ of $E$ such that $|E'|>|E|-\epsilon/2$
and $f': E'\to\C$ is (uniformly) continuous.
Fix an angle $\beta\in(0,\pi/2)$ and a radius $r\in (\sin\beta, 1)$.
Given $X\subset S^1$, let
$$W_X=\cup_{a\in X}\Delta_a(\beta, r)$$
(see Definition \ref{d-triangle})).
For every $n$ big enough, define $M_n: E'\to\R$ as follows:
$$M_n(a)=\sup\{|f'(z)-f'(a)|: z\in \Delta_a(\beta, 1-1/n)\}.$$
Since $f'(z)\to f'(a)$ as $z\to a$ while $|\arg(z-a)|<\beta$,
$M_n(a)\to 0$ for every $a\in E'$. By Egorov's theorem, there exists a closed set $L\subset E'$, $|L| > |E'|-\epsilon/2$,
such that $M_n\to 0$ uniformly on $L$. (Note that $|L|>|E|-\epsilon$.) Consider $f'$ on $W_L$.
We show that $f'(z)\to f'(a)$ as $z\to a$ in $W_L$ uniformly in $a\in L$. Indeed, given a small $t>0$, choose a small $\delta_0>0$ such that $|f'(b)-f'(a)|<t$ whenever $a,b\in L$ and $|b-a|<\delta_0$, and choose $N$ such that $M_N(b)<t/2$ for all $b\in L$.
Now, let
$\delta=\min\{1/N, (\delta_0/2)\cos\beta\}$.
It is chosen so that,
for any $z\in B(1,\delta)\cap\D$, if $z\in\Delta_c(\beta, r)$ for some $c\in S^1$ then
$|c-1|<\delta_0$.


If now $|z-a|<\delta$ for some $a\in L$, $z\in W_L$, then, by the definition of $W_L$,
$z\in\Delta_b(\beta, r)$
for some $b\in L$, on the other hand,
$|b-a|<\delta_0$, by the choice of $\delta$. Therefore,
$$|f'(z)-f'(a)|\le |f'(z)-f'(b)|+|f'(b)-f'(a)|< M_N(b)+t/2<t/2+t/2=t.$$
Thus
$f'(z)\to f'(a)$ as $z\to a$ for $z\in W_L$, uniformly in $a\in L$. Finally, as $f'$ is continuous in $\D$, we conclude that $f'$ is a continuous (in particular, bounded) function in the closure $\overline{W_L}$ of $W_L$.
Let $X=\overline{W_L}\cup \{|z|\le r\}$. Then
$X\subset\D\cup L$
and $\partial X$ is a Jordan curve which consists of $\partial{W_L}$ completed by arcs of the circle $\{|z|=r\}$. Note that
$X\cap S^1=L$.

{\bf Claim.} {\it $f:\partial X\to\C$ is injective.}

Indeed, since $f:\D\to\D$ is univalent and $f(E)\subset S^1$, it remains to show that, if $Z=f(\zeta_1)=f(\zeta_2)$ for some
$\zeta_1,\zeta_2\in L$, then $\zeta_1=\zeta_2$. Assume the contrary. As the angular derivative $f'(\zeta)$ exists and not $0,\infty$ for
every $\zeta\in E$, $f$ is conformal at every "triangle" $\Delta_\zeta(\beta', r')$, for all $\beta'\in(0,\pi/2)$ and all $r'$ small enough:
if $\gamma_1,\gamma_2$ are any two smooth curves in $\Delta_\zeta(\beta',r')\cup\{\zeta\}$ intersecting at the point $\zeta$ at an angle $\alpha$, then
$f(\gamma_1)$, $f(\gamma_2)$ are two smooth curves in $\D\cup f(\zeta)$ that intersects at the same angle $\beta'$.
Fix $\beta'\in (\pi/4,\pi/2)$ and apply this property to $\Delta_{\zeta_i}(\beta',r')$, $i=1,2$ and $r'>0$ small enough.
We obtain two small disjoint (open) triangles $T_1, T_2$ in $\D$ each having an angle $>\pi/2$ at their common vertex $Z=f(\zeta_1)=f(\zeta_2)$,
a contradiction, which proves the claim.


In view of the Claim,
it would be enough to prove (\ref{eq-bdist}) with $L$ replaced by $F\subset L$ such that $W_F$ is a component of the (open) set $W_E$. Fix such an $F$. $W_F$ is a simply connected domain, hence, there exists a conformal isomorphism $R:\D\to W_F$.
As $W_F$ is  bounded by a rectifiable Jordan curve, by Riesz-Privalov's theorem, the derivative
$R'$ belongs in the Hardy space $H^1(\D)$. Hence, $R^{-1}(F)$ is measurable.
Consider a univalent map $g:=f\circ R:\D\to\D$. As $f'$ is bounded in $W_F$,
$g'\in H^1(\D)$ as well and the set $g(R^{-1}(F))$ is measurable.

On the other hand, by the Claim, $g(\D)$ is a Jordan domain. Therefore,
by Riesz-Privalov's theorem, we can write:
$$|f(F)|=|g(R^{-1}(F))|=\int_{R^{-1}(F)}|g'(w)||dw|=\int_{R^{-1}(F)}|f'(R(w))||R'(w)||dw|=$$
$$=\int_F |f'(z)||dz|\ge |F| \sqrt{\frac{(1-|f(0)|)^3}{(1+|f(0)|)|f'(0)|}},$$
where the latter inequality follows from Lemma \ref{l-bounds}.

\end{proof}

\begin{coro}\label{c-bdist}
Let $f:\D\to\D$ be univalent. Let $E\subset S^1$ be a measurable set such that $f(E)\subset S^1$, and $|E|>0$.
Then
\begin{equation}\label{eq-c-bdist}
|E| \sqrt{\frac{(1-|f(0)|)^3}{(1+|f(0)|)|f'(0)|}}\le 2\pi.
\end{equation}
\end{coro}

\

\

For the sake of completeness, we conclude this appendix with a direct short "dynamical" proof of (\ref{eq-b2}) (in fact, a slightly stronger one, see (\ref{eq-dyn})) assuming that there exists
an open arc $I\subset S^1$, such that $\zeta\in I$ and $f(I)\subset S^1$. In this case $f$ extends to a holomorphic function through the arc $I$,
in particular, $1\le |f'(\zeta)|<\infty$.
Replacing $f(z)$ by $f(\zeta z)/f(\zeta)$ one can assume w.l.o.g. that $\zeta=f(\zeta)=1$. Then $f'(1)$ is real positive.

Recall \cite{Ahl} that, given a family $\Gamma$ of curves in the plane,
the extremal length of $\Gamma$,
$$\lambda(\Gamma)=\sup_{\nu}\frac{L(\nu)}{A(\nu)},$$
over all measurable
$\nu:\R^2\to \{x\ge 0\}$ such that
$A(\nu)=\int\int\nu^2 dxdy \neq 0,\infty$, and,
for each such $\nu$,
$L_\gamma(\nu)=\int_\gamma \nu |dz|$,
and
$L(\nu)=\inf_{\gamma\in \Gamma}L_\gamma(\nu)$.
We will use two easy properties of the extremal length: it is a conformal invariant, and $\lambda(\Gamma)\ge \frac{L(\nu)}{A(\nu)}$, for every admissible function $\nu$.
\begin{lemma}\label{l-extr}
\begin{equation}\label{eq-dyn}
\log\frac{1}{|f'(0)|}\le
\frac{|L|^2}{\Re(L)}\le
2\log f'(1),
\end{equation}
for some branch $L$ of $\log\frac{1}{f'(0)}$.
\end{lemma}
\begin{proof}(cf. \cite{l-mult})
Let $\alpha=f'(0)$, $\rho=f'(1)$, so that $0<|\alpha|<1$ and $1\le\rho<\infty$.
Case $\rho>1$. By Koenigs' linearization theorem \cite{CG} applied to $f$ near its repelling fixed point $1\in S^1$, $f:\D\to\D$ in a semi-neighborhood $\Omega_1\subset\D$ of the point $1$ is conformally conjugate to a linear map $M_1: w\mapsto \rho w$ in a semi-disk $\D_\epsilon^+:=\{|w|<\epsilon, \Im(w)>0\}$. Hence,
$\Omega_1$ contains a family $\Gamma$ of (images of) curves which, after projection to $\D_\epsilon^+$, is a family $\Gamma_1$
of all curves $\gamma_w$ such that
$\gamma_w\subset \D_\epsilon^+\setminus \rho^{-1}\overline{\D_\epsilon^+}$ and $\gamma_w$ joins some $w$ with $M_1^{-1}(w)=\rho^{-1} w$, over all $|w|=\epsilon$, $\Im(w)>0$.


It is well-known \cite{Ahl} that $\lambda(\Gamma_1)=\log\rho/\pi$.

On the other hand, $\D$ is the basin of attraction of the attracting fixed point $0$ of $f:\D\to\D$. By the same Koenigs' theorem,
$f:\D\to\D$ is conformally conjugate to a linear map $M_0: w\mapsto \alpha w$ where $M_0: \Omega\to\Omega$ for some simply-connected domain $\Omega\ni 0$.
Then $\Gamma$ is projected to a family of curves $\Gamma_0$ which fill some domain $U_0\subset\Omega\setminus\{0\}$ as follows:
each $\gamma\in\Gamma_0$ joins some $w\in\partial U_0$ to $M_0(w)=\alpha w$, and $M_0^{\pm 1}(w)\notin U_0$ whenever $w\in U_0$.
Let $\nu_0(w)=1/|w|$ for $w\in U_0$ and $\nu_0=0$ otherwise. It is easy to check that $L_{\nu_0}(\gamma)\ge |\log\alpha^{-1}|$
for all $\gamma\in\Gamma_0$, and,
since $\nu_0(w)|dw|$ is invariant by $M_0$, $A(\nu_0)\le 2\pi\log|\alpha^{-1}|$.
Thus
$$\frac{\log\rho}{\pi}=\lambda(\Gamma_1)=\lambda(\Gamma)=\lambda(\Gamma_0)\ge \frac{L(\nu_0)^2}{A(\nu_0)}\ge\frac{|\log\alpha^{-1}|^2}
{2\pi \log|\alpha^{-1}|}.$$
This proves (\ref{eq-dyn}) in the case $\rho>1$.
The remaining case $\rho=1$ is impossible.
Indeed, let us consider the quotient space $\tilde f_1:=\Omega_1/\{z\sim f^{-1}(z)\}$.
If $\rho>1$, we have seen that $\tilde f_1$ is conformally equivalent to a finite geometric cylinder $\D_\epsilon^+/\{w\sim \rho^{-1}w\}$ and $\Gamma$ is just the set of all closed curves wrapping once around $\tilde f_1$. On the other hand, if $\rho=1$,
then, by the Leau-Fatou Parabolic linearization theorem (see e.g. \cite{CG}), $\tilde f_1$ is conformally equivalent to a geometric cylinder
$\{Re(w)>0\}/\{w\sim w+1\}$
of infinite hight so that $\lambda(\Gamma)=0$. Thus $0=\lambda(\Gamma)=\lambda(\Gamma_0)\ge\frac{|\log\alpha^{-1}|^2}{2\pi \log|\alpha^{-1}|}>0$,
a contradiction.
\end{proof}


\end{document}